\documentclass[11pt,a4paper]{amsart}
\usepackage{amssymb,amsmath,epsfig,graphics,mathrsfs,enumerate,verbatim}
\usepackage[pagebackref,colorlinks=true,linkcolor=blue,citecolor=blue]{hyperref}
\usepackage{fancyhdr}
\usepackage{hyperref}
\hypersetup{
 colorlinks   = true,
 urlcolor     = blue,
 linkcolor    = blue,
 citecolor   = red ,
 bookmarksopen=true
}

\usepackage{amsmath}
\usepackage{amsfonts}
\usepackage{amssymb}
\usepackage{amsthm}
\usepackage{epsfig,graphics,mathrsfs}
\usepackage{graphicx}
\usepackage{dsfont}

\usepackage[usenames, dvipsnames]{color}

\usepackage{hyperref}

\def \N {\mathbb{N}}

\def \phi {\varphi}

\def \R {\mathbb{R}}

\def \G{\Gamma}
\newcommand{\Ba}{\mathcal B_{a}}

\def \vf{\varphi}

\newcommand{\p}{\partial}

\newcommand{\la}{\lambda}

\newcommand{\Rd}{\mathbb R^d}

\numberwithin{equation}{section}

\newcommand{\beq}{\begin{equation}}
\newcommand{\bea}[1]{\begin{array}{#1} }
\newcommand{\eeq}{ \end{equation}}
\newcommand{\ea}{ \end{array}}

\newcommand{\ve}{\varepsilon}

\newcommand{\sa}{\langle}
\newcommand{\da}{\rangle}

\newcommand{\C}{\mathbb{C}}

\newcommand{\sn}{\mathscr S^+}
\newcommand{\snn}{\mathbb S^+}

\newtheorem{theorem}{Theorem}[section]
\newtheorem{lemma}[theorem]{Lemma}
\newtheorem{proposition}[theorem]{Proposition}
\newtheorem{corollary}[theorem]{Corollary}

\newtheorem{definition}[theorem]{Definition}

\numberwithin{equation}{section}

\begin{document}

\title[]{Strichartz estimates for a Schr\"odinger equation on the half-line with a Neumann boundary condition}



\date{}

\begin{abstract}
In this paper we prove some new Strichartz estimates related to the Cauchy problem for the Bessel operator on the half-line and we  establish a fractal version of the Tomas-Stein restriction theorem for the Hankel transform.  
Then we use the proved Strichartz  estimates to show  global in time well-posedness  for  a class of  nonlinear $L^2_a$-critical problems, and local in time well-posedness in the sub-critical case. \end{abstract}

\author{Nicola Garofalo}
\address{School of Mathematical and Statistical Sciences\\ Arizona State University}\email[Nicola Garofalo]{nicola.garofalo@asu.edu}

\author{Gigliola Staffilani}
\address{Department of Mathematics\\ Massachusetts Institute of Technology\\
Cambridge, Massachusetts, 02139\\USA}\email[Gigliola Staffilani]{gigliola@mit.edu}

\maketitle


\section{Introduction}\label{S:intro}

In this paper we establish a priori estimates of Strichartz type for the following singular Cauchy problem on the half-line $\R^+= (0,\infty)$ with Neumann boundary condition
\begin{equation}\label{cp0}
\begin{cases}
\p_t u - i \left(\p_{xx} u + \frac a{x} \p_x u\right) = F(x,t),\ \ \ \ \ x\in \R^+,\ t\in \R,\ \ a >-1, 
\\
\underset{x\to 0^+}{\lim} x^a \p_x u(x,t) = 0,\ \ \ \ t\in \R,
\\
u(x,0) = \vf(x).
\end{cases}
\end{equation}
When $a\geq 0$ and $F = \mu |u|^{p-1} u$, with $\mu\in \C$ and $p>1$, we prove  global in time existence and uniqueness for small initial data when the problem is $L^2_a$-critical (i.e., $p = 1+\frac{4}{a+1}$). In the $L^2_a$-subcritical case (i.e., $p < 1+\frac{4}{a+1}$), we obtain existence and uniqueness on a small time interval for any $\mu\in \C$, and we are able to upgrade this local result to a global in time one when $\mu = i b$, with $b\in \R$.
  
The background motivation for the present study is our upcoming work \cite{GaS} on nonlinear Schr\"odinger equations with memory. To explain the context, we mention that in the recent paper \cite{HOS}, see also \cite{OST},
the authors study the Cauchy problem in the half-plane $\R^2_+ = \R^+_x\times\R_y $ with nonlinear Neumann condition
\begin{equation}\label{ost}
\begin{cases}
\p_t U - i (\p_{yy} U + \p_{xx} U) = f(x,y,t),\ \ \ \ \ \ \ \ \ (y,x)\in \R^2_+, t>0,
\\
\underset{x\to 0^+}{\lim} \p_x U(\cdot,x,\cdot)  = -\mu |U|^{q-1} U,\ \ \ \ \ \ \ \ \ \ \ \ \ \ \ y\in \R, t>0,
\\
U(x,y,0) = U_0(x,y),
\end{cases}
\end{equation}
where $q>1$ and $\mu\in \C\setminus\{0\}$ are given constants. When $f\equiv 0$, the problem \eqref{ost} can be interpreted as a suitable ``extension" of the following nonlocal Cauchy problem with memory
\begin{equation}\label{ost2}
\begin{cases}
(\p_t u - i \p_{yy} u)^{1/2}  = \mu |u|^{q-1} u,\ \ \ \ \ \ \ \ \ y\in \R, t>0,
\\
u(y,t) = u_0(y,t),\ \ \ \ \ \ \ \ \ \ \ \ \ \ \ \ \ \ \ \ \ \ \ \ y\in \R, t\le 0.
\end{cases}
\end{equation}
More in general, in \cite{GaS} we study the Cauchy problem 
\begin{equation}\label{ost3}
\begin{cases}
(\p_t u - i \Delta_y u)^{s}  = \mu |u|^{q-1} u,\ \ \ \ \ \ \ \ \ \ \ y\in \Rd, t>0,
\\
u(y,t) = u_0(y,t),\ \ \ \ \ \ \ \ \ \ \ \ \ \ \ \ \ \ \ \ \ \ \ y\in \Rd, t\le 0,
\end{cases}
\end{equation}
where $s\in (0,1)$. 
To \eqref{ost3} we associate the \emph{extension problem} in the upper half-space $\R^{d+1}_+ = \R^+_x \times \Rd_y$
\begin{equation}\label{cp02}
\begin{cases}
\p_t U - i \left(\Delta_y U+ \p_{xx} U + \frac a{x} \p_x U\right) = 0,\ \ \ \ \ (x,y)\in \R^{d+1}_+,\  t\in \R,
\\
\underset{x\to 0^+}{\lim} x^a \p_x U(\cdot,x,\cdot) = -\mu |U|^{q-1} U,\ \ \ \ t\in \R,
\\
U(x,y,0) = U_0(x,y),
\end{cases}
\end{equation}
where $a = 1-2s\in (-1,1)$ (note that for \eqref{ost}, \eqref{ost2} we have $s=1/2$ if and only if  $a = 0$). 
This scenario is clearly reminiscent of the celebrated extension procedure of Caffarelli and Silvestre for the fractional powers of the Laplacian $(-\Delta)^s$, see \cite{CS}. As it is well-known, their work has had a transformative impact in the analysis of nonlocal problems, particularly in the study of free boundaries. The problem \eqref{cp02} is, however, considerably different since the extension operator is itself of Schr\"odinger type and it thus involves oscillatory phenomena in which positivity plays no role. For related works involving the wave operator, we refer the reader to \cite{KST, EGV}.

The study of \eqref{cp02} brings us to the model problem \eqref{cp0}, since understanding the latter is critical to solving the former. For instance, if with $\mu = 0$ we look for solutions of \eqref{cp02} which do not depend on the variable $y\in \Rd$, we are led to \eqref{cp0}. Henceforth, we denote by
\begin{equation}\label{Ba}
\Ba u = \p_{xx} u + \frac a{x} \p_x u = x^{-a} \p_x(x^a \p_x)
\end{equation}
the Bessel operator on the half-line $\R^+$. As it is well-known, with its associated invariant measure $d\omega_a(x) = x^a dx$, such operator
plays a central role in analysis and geometry, particularly in the study of partial differential equations and stochastic processes in which symmetries are involved. For instance, the Laplacian in $\Rd$ is given by
\begin{equation}\label{delta}
\Delta = \mathcal B_{d-1} + \frac{1}{r^2} \Delta_{\mathbb S^{d-1}},
\end{equation}
where $\Delta_{\mathbb S^{d-1}}$ is the Laplace-Beltrami operator on the unit sphere $\mathbb S^{d-1}$. The operator $-\Ba$ in \eqref{Ba} is nonnegative on $L^2_a = L^2(\R^+, d\omega_a)$ when restricted to $C^\infty_0(\R^+)$, as one has
\begin{equation}\label{pos}
\sa -\Ba \psi,\psi\da = - \int_0^\infty \bar \psi\ \Ba \psi\  d\omega_a = \int_0^\infty |\p_x \psi|^2 d\omega_a.
\end{equation}
The weighted Hardy inequality
\[
\int_{0}^\infty  \frac{|\psi|^2}{x^2}  d\omega_a \le \left(\frac{2}{|1-a|}\right)^2 \int_0^\infty |\p_x \psi|^2 d\omega_a,
\]
provides a lower bound for the Dirichlet form associated with the diffusion operator $-\Ba$.
Since the literature on \eqref{Ba} is vast, we only provide few references which are more closely connected to our study: \cite{We1, We, We2, MS, MO, BoSa, CS, BGL, Gft, Gams}.

The end-points $x = 0$ and $\infty$ of the half-line $\R^+$ are singular for $\Ba$. Frobenius' indicial equation $\nu(\nu-1) + a \nu = 0$ shows that two linearly independent solutions of $\Ba u = 0$ are $u_1(x) = 1$ and $u_2(x) = x^{1-a}$ (when $a =1$ one has to replace such $u_2$ with the singular solution $u_2(x) = \log x$). It is easily checked that, for any $\ve>0$, one has $u_1\in L^2_a(0,\ve) = L^2((0,\ve),d\omega_a)$ if and only if $a>-1$ and $u_2\in L^2_a(0,\ve) $ if and only if $a<3$ ($u_1(x)=1$ and $u_2(x) = \log x$ are both in $L^2_1(0,\ve)$ when $a=1$). Therefore, according to Weyl's characterization of singular Sturm-Liouville problems, $x = 0$ is limit-circle if and only if $-1<a<3$, see \cite[Chap. 2]{Ti} and \cite[Chap. 13]{DS}. The point $x = \infty$, instead, is always limit-point. We infer that 
\begin{equation}\label{sa}
\Ba\ \text{is self-adjoint in}\ L^2_a\ \ \Longleftrightarrow\  
a\in (-\infty,-1] \cup [3,\infty).
\end{equation}
We stress that for $a$ in the range in \eqref{sa} no boundary conditions are necessary. If instead $-1<a<3$, we impose in \eqref{cp0} the Neumann condition 
\begin{equation}\label{neu}
\underset{x\to 0^+}{\lim}\ x^a \p_x u = 0.
\end{equation}
This serves to eliminate the solution $u_2(x) = x^{1-a}\in L^2_a(0,\ve)$ (or $u_2(x) = \log x$ when $a = 1$). Therefore, 
\begin{equation}\label{saa}
-1 < a < 3\ +\ \eqref{neu}\ \Longrightarrow\ \Ba\ \text{is self-adjoint in}\ L^2_a.
\end{equation}
It is also worth noting here that, if $a\ge 1$, for no $\ve>0$ the function $u_2(x) = x^{1-a}$ belongs to the energy space
\[
H^1_a(0,\ve) = \{\vf\in L^2_a(0,\ve)\mid \p_x \vf\in L^2_a(0,\ve)\}.
\]
Therefore, when $a\ge 1$ such $C^\infty(\R^+)$ function is \emph{not} an energy solution of $\Ba u = 0$ (although it is a strong solution to the equation).

Henceforth, when we write $\Ba$ we will abuse the notation and indicate the self-adjoint extension of \eqref{Ba} in $L^2_a$, with the understanding that, when $-1<a<3$, we impose the condition \eqref{neu}. With this agreement, for any $a>-1$ the Stone-von Neumann theorem guarantees that $e^{i t \Ba}$ is a well-defined unitary group on $L^2_a$. 
The domain of such group will be the space $\sn$ of all functions $\vf:\R^+\to \C$ such that $\vf\in C^\infty(\R^+)$ and for every $m, k\in \mathbb N_0$ one has
\[
\gamma_{m,k}(\vf) :=  \underset{x>0}{\sup} \left|x^m \left(\frac 1x \p_x\right)^k \vf(x)\right| < \infty.
\]
The family $\gamma_{m,k}$ is a countable collection of seminorms which generates a Frechet space topology on $C^\infty(\R^+)$. Note that since $\gamma_{0,1}(\vf)<\infty$, we have $|\p_x \vf(x)|\le \gamma_{0,1}(\vf) x$ for any $x>0$.  This guarantees in particular that any function $\vf\in \sn$ satisfies the Neumann condition \eqref{neu} when $a>-1$. In analogy to $\sn$, we consider the space $\mathbb S^+$ of functions $F:\R^+\times \R\to \C$ in $C^\infty(\R^+\times \R)$ such that for any $p\in \mathbb N_0$ one has
\[
||F||_{p} = \underset{k+\ell \le p}{\sup}\ \ \underset{(x,t)\in\R^+\times \R}{\sup} (1+x^2+t^2)^{\frac p2} \left|\left(\frac 1x \p_x\right)^k \p_t^\ell F(x,t)\right|<\infty.
\]
Note that for any $F\in \snn$ we have for $a>-1$
\[
\underset{t\in \R}{\sup} |x^a \p_x F(x,t)|\ \underset{x\to 0^+}{\longrightarrow}\ 0.
\]
We have the following proposition that will be proved in Section 2. 
\begin{proposition}\label{P:rep}
Given $\vf\in \sn$ and $F\in \snn$, the solution of the Cauchy problem \eqref{cp0} admits the representation 
\begin{equation}\label{kernel}
u(x,t) =  \int_0^\infty S_a(x,y,t) \vf(y) d\omega_a(y) + \int_0^t \int_0^\infty S_a(x,y,t-s) F(y,s) d\omega_a(y) ds,
\end{equation}
where for $x, y>0$ we have defined
\begin{equation}\label{SaS0}
S_a(x,y,t) := 
\begin{cases}
\frac{e^{i \frac{(a+1) \pi}{4}}}{(2|t|)^{\frac{a+1}2}}  \left(\frac{xy}{2|t|}\right)^{\frac{1-a}2}J_{\frac{a-1}2}(\frac{xy}{2|t|}) e^{-i \frac{x^2+y^2}{4|t|}},\ \ \ \ \ t<0,
\\
\frac{e^{-i \frac{(a+1) \pi}{4}}}{(2|t|)^{\frac{a+1}2}}  \left(\frac{xy}{2|t|}\right)^{\frac{1-a}2}J_{\frac{a-1}2}(\frac{xy}{2|t|}) e^{i \frac{x^2+y^2}{4|t|}},\ \ \ \ \ t>0,
\end{cases}
\end{equation}
 where $J_\nu(z)$ denotes the Bessel function of the first kind and order $\nu$, see \eqref{besseries} below.
\end{proposition}

We note the following properties:
\begin{itemize}
\item[(i)] $S_a(x,y,t) = S_a(y,x,t)$;
\item[(ii)] $S_a(\la x,\la y,\la^2 t) = \la^{-(a+1)} S_a(x,y,t)$,\ \ \ \ $\la>0$;
\item[(iii)] $\int_0^\infty S_a(x,y,t) d\omega_a(y) = 1$\ \ \ \ $-1<a<2$\ \ (as a generalized Riemann integral).
\end{itemize}
Property (i) is obvious from \eqref{SaS0} and it derives from the symmetry of the operator $\Ba$ in $L^2_a$. Property (ii) reflects the invariance of $\p_t - i \Ba$ with respect to the scalings $\la \to (\la z,\la^2 t)$. In particular, (ii) implies that for $t\not=0$
\begin{equation}\label{scaleinv}
S_a(x,y,|t|) = |t|^{-\frac{a+1}2} S_a(\frac{x}{\sqrt{|t|}},\frac{y}{\sqrt{|t|}},1).
\end{equation}
Property (iii) is proved in Proposition \ref{P:1}. In contrast with (iii), we note that for the kernel $p_a(x,y,t)$ of the heat semigroup $e^{-t\Ba}$ with the Neumann condition \eqref{neu} (see also \eqref{heatsg} below), one has  the stochastic completeness
\[
\int_0^\infty p_a(x,y,t) d\omega_a(y) = 1,\ \ \ \ \ \ \ \ x, t>0,
\]
in the full range $a>-1$, see \cite[Prop. 2.3]{Gams}. This difference is subtle, but one can surmise that the more favorable situation of the heat semigroup is due to that fast decay of its kernel at infinity. 
From the well-known properties of the Bessel functions (see \eqref{Js} below), we see that for any $y>0$,
\begin{equation}\label{pazero}
S_a(0,y,t) = \underset{x\to 0^+}{\lim} S_a(x,y,t) = \begin{cases} \frac{e^{i \frac{(a+1) \pi}{4}}}{2^{a} \G(\frac{a+1}2)} |t|^{-\frac{a+1}{2}} e^{-i\frac{y^2}{4|t|}},\ \ \ \ t<0,
\\
\frac{e^{-i \frac{(a+1) \pi}{4}}}{2^{a} \G(\frac{a+1}2)} |t|^{-\frac{a+1}{2}} e^{i\frac{y^2}{4|t|}},\ \ \ \ t>0,
\end{cases}
\end{equation}
but, in the regime $-1<a<0$, the behavior of $S_a(x,y,t)$ is dramatically different when $x\not=0$.
We also have for any $y>0$ and $t\not=0$,
\begin{equation}\label{dzpazero}
\underset{x\to 0^+}{\lim} x^a \p_x S_a(x,y,t) = 0.
\end{equation}
Since when $a>-1$ the kernel $S_a(x,y,t) $ is the fundamental solution for problem \eqref{cp0}, the limit relation \eqref{dzpazero} should come as no surprise.

In Section 2,  Proposition \ref{P:cpH} also reveals the important link between the Cauchy problem \eqref{cp0} and the modified Hankel transform $\mathcal H_\nu$ in \eqref{ht1}. This link is fundamental to prove the main estimates in this work.

To state our results, we need to introduce some notation. For any $1\le r<\infty$, we will indicate by $L^r_a$ the Banach space of the measurable functions $f: \R^+\to \overline \C$ such that 
\[
||f||_{L^r_a} = \left(\int_0^\infty |f(x)|^r d\omega_a(x)\right)^{1/r} < \infty.
\] 
Given exponents $1\le q, r<\infty$, we will also use the mixed-norm Lebesgue spaces $L^q_t L^r_a$ of measurable functions $u:\R^+\times \R\to \overline \C$ such that 
\[
||u||_{L^q_t L^r_a} = ||\ ||F(\cdot,t)||_{L^r_a}\ ||_{L^q_t} = \left(\int_\R ||u(\cdot,t)||^q_{L^r_a} dt\right)^{1/q}<\infty.
\] 
For any given $\vf\in \sn$ denote by $u$ the unique solution of the problem \eqref{cp0} with zero  term, namely  $F\equiv 0$, and  given by \eqref{kernel}. Assume that there exist exponents $1\le q, r\le \infty$ such that for some universal constant $C = C(a,r,q)>0$, the following a priori estimate holds
\begin{equation}\label{strinec}
||u||_{L^q_t L^r_a} \le C \ ||\vf||_{L^2_a}.
\end{equation}  
Now, for any $\la >0$ the function $u_\la(x,t) = u(\la x, \la^2 t)$ also solves  \eqref{cp0} with $F = 0$, and initial datum $\vf_\la(x) = \vf(\la x)$. If we substitute such $u_\la, \vf_\la$ in \eqref{strinec}, after an elementary change of variable we obtain
\[
\la^{-(\frac{a+1}r + \frac 2q)} ||u||_{L^q_t L^r_a} \le C\ \la^{-\frac{a+1}2} ||\vf||_{L^2_a}.
\]
We thus reach the conclusion that a necessary condition for \eqref{strinec} to hold is that
\begin{equation}\label{ad0}
\frac 2q = (a+1)\left(\frac 12 - \frac 1r\right).
\end{equation}
Note that this equation forces $r\ge 2$. 
\begin{definition}\label{D:a}
Let $a>-1$. Let $r>2$, and assume furthermore that $r<\frac{2(a+1)}{a-1}$ if $a>1$. We say that a pair $(q,r)$ is $a$-\emph{admissible} if  
\[
\frac 2q = (a+1)(\frac 12 - \frac 1r).
\]
\end{definition}

Our first main result shows that \eqref{ad0} is also sufficient for \eqref{strinec}, at least in the regime $a\ge 0$. These are indeed the Strichartz estimates
for our Cauchy problem  \eqref{cp0} that correspond to the ones for the Schr\"odinger operator for example in \cite{Caze} and also \cite[Chap. 11]{MuS1}. More precisely, we have the following.

\begin{theorem}\label{T:GV}
Let $a\ge 0$ and let the pair $(q,r)$ be $a$-\emph{admissible}. There exists $C(a,r)>0$ such that for every $\vf\in \sn$ and $F\in \snn$, the unique solution $u$ of the problem \eqref{cp0}, satisfies the following a priori estimate
\begin{equation}\label{striuno}
||u||_{L^q_t L^r_a} \le C(a,r) \left(||\vf||_{L^2_a} + ||F||_{L^{q'}_t L^{r'}_a}\right).
\end{equation}
\end{theorem}
From Theorem \ref{T:GV} we have the following corollary (see the arguments in \cite{Caze}).
\begin{corollary}\label{C:caze}
Let $a\ge 0$ and let the pairs $(q,r)$ and $(\tilde q,\tilde r)$ be $a$-\emph{admissible}. There exists $C(a,r, \tilde r)>0$ such that for every $\vf\in \sn$ and $F\in \snn$, the unique solution $u$ of the problem \eqref{cp0}, satisfies the following a priori estimate
\begin{equation}\label{striuno}
||u||_{L^q_t L^r_a} \le C(a,r) \left(||\vf||_{L^2_a} + ||F||_{L^{\tilde q'}_t L^{\tilde r'}_a}\right).
\end{equation}
\end{corollary}

As a special case of Theorem \ref{T:GV}, we note that  since for any $r>2$ the pair $(r,r)$ is $a$-admissible when
\begin{equation}\label{qr}
r = \frac{2(a+3)}{a+1},\ \ \ \ \ \ r' = \frac{2(a+3)}{a+5},
\end{equation}  
 we obtain the following result, reminiscent of the {\it classical} Strichartz estimate in \cite[Cor.1]{Stri}
\begin{equation}\label{stri}
||u||_{L^{\frac{2(d+2)}d}(\R^{d+1})} \le C(d) \left(||\vf||_{L^2(\Rd)} + ||F||_{L^{\frac{2(d+2)}{d+4}}(\R^{d+1})}\right),
\end{equation}
if we formally let $d = a+1$. 
We write this fact in the following corollary.
\begin{corollary}\label{T:stribest0}
Given $a\ge 0$, there exists $C(a)>0$ such that, for every $\vf\in \sn$ and $F\in \snn$, the solution $u$ of \eqref{cp0} satisfies the a priori inequality
\begin{equation}\label{stridue0}
\left(\int_{\R} \int_0^\infty |u(x,t)|^{\frac{2(a+3)}{a+1}} d\omega_a(x) dt\right)^{\frac{a+1}{2(a+3)}} \le C(a) \left\{||\vf||_{L^2_a} + \left(\int_{\R} \int_0^\infty |F(x,t)|^{\frac{2(a+3)}{a+5}} d\omega_a(x) dt\right)^{\frac{a+5}{2(a+3)}}\right\}.
\end{equation}
 \end{corollary}
 
To state our second result, let $\nu\in \C$ be such that $\Re\nu>-1$. For a function $u:\R^+\times \R\to \C$ we introduce its \emph{Fourier-Hankel transform} of $u$, as follows:
\begin{equation}\label{FH}
\widetilde{\mathcal H}_\nu(u)(\xi,\tau) = \int_{\R} e^{-2\pi i \tau t} \mathcal H_\nu(u(\cdot,t))(\xi) dt,
\end{equation}
where we have denoted by $\mathcal H_\nu(u(\cdot,t))(\xi)$ the modified Hankel transform of $u$ with respect to the variable $x>0$, see \eqref{ht1}.
Then we have the following result.

\begin{theorem}[Restriction estimates]\label{T:restriction}
Under the assumptions of Theorem \ref{T:GV}, we have for every $F\in \snn$,
\begin{equation}\label{re1}
\left(\int_0^\infty |\widetilde{\mathcal{H}}_{\frac{a-1}2}(F)(\xi,-\frac{\xi^2}{2\pi})|^2 d\omega_a(\xi) \right)^{1/2}  \le C(a,r)\ ||F||_{L^{q'}_t L^{r'}_a}.
\end{equation}
In particular, when $q'=r'$ is given by \eqref{qr}, we obtain the following fractal version of the \emph{Tomas-Stein restriction inequality} for the half-parabola
$P = \{(\xi,\tau)\in \R^2\mid \tau = - 2\pi \xi^2,\ \xi\ge 0\}$
\begin{equation}\label{re2}
\left(\int_0^\infty |\widetilde{\mathcal{H}}_{\frac{a-1}2}(F)(\xi,-\frac{\xi^2}{2\pi})|^2 d\omega_a(\xi) \right)^{1/2}  \le C(a)\ \left(\int_{\R} \int_0^\infty |F(x,t)|^{\frac{2(a+3)}{a+5}} d\omega_a(x) dt\right)^{\frac{a+5}{2(a+3)}}.
\end{equation}
\end{theorem}
Concerning Theorem \ref{T:restriction}, we mention that if 
\[
\Sigma = \{(\xi,\tau)\in \R^2\mid \tau = f(\xi),\ \ \xi\in [a,b]\},
\]
and we denote by 
\[
ds_a := \sqrt{1+f'(\xi)^2}\ d\omega_a(\xi)
\]
the \emph{weighted arc-length measure} on $\Sigma$, if for $M>0$ we consider the truncated parabola
\[
\Sigma_M = \{(\xi,\tau)\in \R^2\mid \tau = - 2\pi \xi^2,\ 0\le \xi\le M\},
\]
then the inequality \eqref{re2} implies the existence of a constant $C(a,M)>0$ such that the following restriction inequality holds
\begin{equation}\label{re3}
\left(\int_{\Sigma_M} |\widetilde{\mathcal{H}}_{\frac{a-1}2}(F)|^2  ds_a \right)^{1/2}  \le C(a,M)\ \left(\int_{\R} \int_0^\infty |F(x,t)|^{\frac{2(a+3)}{a+5}} d\omega_a(x) dt\right)^{\frac{a+5}{2(a+3)}}.
\end{equation}

\vskip 0.2in

Theorems \ref{T:GV} and \ref{T:restriction} and Corollary \ref{T:stribest0} cover the range $a\ge 0$. When $-1<a<0$ these results are no longer valid, but we have instead the following weighted Strichartz estimates. Hereafter, we denote by $k_1$ the weight $k_1(x) = \min\{1,x^{\frac a2}\}$ (see Definition \ref{D:mn} below, and also Section \ref{Not} for the relevant Lebesgue spaces). 

\begin{theorem}\label{T:GV2}
Let $-1<a< 0$ and for $r>2$ assume that $q$ and $q_\infty$ satisfy the conditions
\begin{equation}\label{wa}
\frac 2q = \frac 12 - \frac 1r,\ \ \ \ \ \ \ \ \frac 2{q_\infty} \le (a+1)\left(\frac 12 - \frac 1r\right).
\end{equation}
There exists $C(a,r,q_\infty)>0$ such that for every $\vf\in L^2_a$ and forcing term $F$, such that $F k_1^{1-\frac{2}{r'}}\in L^{q'\cap q'_\infty}_t L^{r'}_a$, the unique solution $u$ of the problem \eqref{cp0} is such that $u(\cdot,t) k_1^{\frac{2}{r}-1}\in L^{q+ q_\infty}_t L^{r}_a$, and satisfies the following weighted a priori estimate
\begin{equation}\label{striunoo}
||\ ||u(\cdot,t) k_1^{\frac{2}{r}-1}||_{L^{r}_a}\ ||_{L_t^{q} + L_t^{q_\infty}} \le C(a,r,q_\infty) \left(||\vf||_{L^2_a} + ||\ ||F(\cdot,t) k_1^{1-\frac{2}{r'}}||_{L^{r'}_a}\ ||_{L_t^{q'}\cap L_t^{q'_\infty}}\right).
\end{equation}
\end{theorem}
We explicitly note that \eqref{wa} and $-1<a<0$ imply
\[
q_\infty \ge \frac{q}{a+1}> q.
\]
We also mention that the presence of the weight $k_1(x)$ and the Lebesgue classes $L^{q'\cap q'_\infty}_t L^{r'}_a$ and $L^{q+ q_\infty}_t L^{r}_a$ (for their description, see Section \ref{Not}) in the statement of Theorem \ref{T:GV2}, are a reflection of the anomalous behavior in the range $-1<a<0$ expressed by the dispersive estimate \eqref{weight} in Proposition \ref{P:dis}. This phenomenon is somewhat similar to that in \cite[Theorem B]{BG}. 

We also have a weighted version of the restriction Theorem \ref{T:restriction}.

\begin{theorem}\label{T:weires}
Let $-1<a< 0$ and for $r>2$ assume that $q$ and $q_\infty$ satisfy the conditions \eqref{wa}.
There exists $C(a,r,q_\infty)>0$ such that for every $F$, for which $F k_1^{1-\frac{2}{r'}}\in L^{q'\cap q'_\infty}_t L^{r'}_a$,
one has the following estimate
\begin{equation}\label{weires}
\left(\int_0^\infty |\widetilde{\mathcal{H}}_{\frac{a-1}2}(F)(\xi,-\frac{\xi^2}{2\pi})|^2 d\omega_a(\xi) \right)^{1/2}   \le C(a,r,q_\infty)  ||\ ||F(\cdot,t) k_1^{1-\frac{2}{r'}}||_{L^{r'}_a}\ ||_{L_t^{q'}\cap L_t^{q'_\infty}}.
\end{equation}
\end{theorem}

\vskip 0.2in

Finally, as a basic application of Theorem \ref{T:GV} we present a well-posedness result for the nonlinear Cauchy problem
\begin{equation}\label{nonlin}
\begin{cases}
\p_t u - i \Ba  u = \mu|u|^{p-1}u,\ \ \ \ \ x\in \R^+,\ t\in \R, 
\\
\underset{x\to 0^+}{\lim} x^a \p_x u(x,t) = 0,\ \ \ \ \ \ \ \ \ t\in \R,
\\
u(x,0) = \vf(x),
\end{cases}
\end{equation}
where $\mu\in \C$ and $p>1$. 

To state the main result we need to introduce the  Banach space in which the solutions to \eqref{nonlin} will live.  For a fixed time interval $[0,T]$ we define the space
\begin{equation}\label{space}
X_T=\{f: [0,T]\times \R \rightarrow \C \mid \|f\|_{L^q_tL^r_a}<\infty\ \text{for every pair}\ (q,r)\ \text{that is}\ a-\text{admissible}\}.
\end{equation}
If $T=\infty$, then we simply write $X$. Endowed with the norm 
\[
\|f\|_{X_T}=\sup_{(q,r)} \|f\|_{L^q_tL^r_a},
\]
the linear space $X_T$ becomes a Banach space. The reader should also see Definition \ref{D:crit} of $L^2_a$-\emph{critical} and \emph{subcritical} nonlinear Cauchy problem \eqref{nonlin}.

\begin{theorem}\label{T:main} Assume $a\ge 0$ and that the initial datum in \eqref{nonlin} is such that $\phi\in L^2_a$. Then: 
\begin{itemize}
\item[(i)] If the Cauchy problem \eqref{nonlin} is $L^2_a$-\emph{critical}, there exists $\ve_0>0$ such that if $\|\phi\|_{L^2_a}\leq \ve_0$ there exists a unique global in time solution $u$ to \eqref{nonlin}, in the sense of Definition \ref{solution},  such that $u\in X\cap C(\R_+, L^2_a)$ and the solution map $\phi\rightarrow u$ is continuous in the appropriate topology.
\item[(ii)] If the Cauchy problem \eqref{nonlin} is $L^2_a$-\emph{subcritical}, for any $\mu\in\C$ and any datum $\phi\in L^2_a$, there exist a time $T=T(\|\phi\|_{ L^2_a}) >0$, and a unique solution $u$ to \eqref{nonlin}, in the sense of Definition \ref{solution}, such that $u\in X_T\cap C([0,T], L^2_a)$, and the solution map $\phi\rightarrow u$ is continuous in the appropriate topology.
If $\mu$ is purely imaginary, then the solution $u$ is global in time, namely for any $M>0$  we have $u\in X_M\cap C([0,M], L^2_a)$.
\end{itemize}
\end{theorem}

The organization of the paper is as follows. In Section \ref{S:solve} we prove Proposition \ref{P:rep}. In Proposition \ref{P:cpH} we highlight the link between the Cauchy problem \eqref{cp0} and the modified Hankel transform \eqref{ht1}. Section \ref{S:distri} is devoted to proving two key results for this paper, the dispersive estimates in Propositions \ref{P:dis} and \ref{P:disp}. These results highlight the different behavior of the unitary group $e^{i t \Ba}$ in the regimes $-1<a<0$ and $a\ge 0$. In Section \ref{S:stri} we prove Theorems \ref{T:GV} and \ref{T:restriction}. In Section \ref{S:weighted} we prove Theorems \ref{T:GV2} and \ref{T:weires}. In Section \ref{S:wp}
 we present an application of the above results to global and local in time  well-posedness for the nonlinear Cauchy problem when the right-hand side in \eqref{cp0} is given by $F = \mu |u|^{p-1} u$, with $\mu\in \C$ and $p>1$. Our main result in this direction is Theorem \ref{T:main}. Section \ref{S:comp} is devoted to a comparison of the Cauchy problem \eqref{cp0} with that for two other Schr\"odinger equations on the half-line: the Kimura type model \eqref{Lb} and the Hardy type operator \eqref{Hm}. The paper closes with an appendix (Section \ref{S:app}) which collects some known material used in the rest of the work.

\subsection{Notation}\label{Not} We adopt the brief notation $L^{q_1+q_2}$ to indicate the
Banach space $L^{q_{1}} + L^{q_{2}}$ of measurable functions $f$ such that $f = f_{1} + f_{2}$, with $f_{1}\in L^{q_{1}}$ and
$f_{2} \in L^{q_{2}}$, endowed with the norm
\begin{align*}
\|f\|_{L^{q_{1}+q_{2}}}=\inf_{f=f_{1}+f_{2},f_{1}\in L^{q_{1}},f_{2} \in L^{q_{2}}}\|f_{1}\|_{L^{q_{1}}}+\|f_{2}\|_{L^{q_{2}}}.
\end{align*}
We also use the short notation $L^{p\cap q}$ to denote the Banach space $L^{p}\cap L^{q}$, endowed with the norm
\begin{align*}
\|f\|_{L^{p\cap q}}=\|f\|_{L^{p}}+\|f\|_{L^{q}}.
\end{align*}
We recall that $(L^{p+q})'=L^{p' \cap q'}$. We respectively indicate by
$L^{q_1+q_2}_t L^r_a$ and $L^{q_1\cap q_2}_t L^r_a$ the space of measurable functions $u:\R^+\times \R\to \overline \C$ such that
\[
||\ ||u(\cdot,t)||_{L^r_a}\ ||_{L_t^{q_1+q_2}}<\infty,\ \ \ \ \ \ \ ||\ ||u(\cdot,t)||_{L^r_a}\ ||_{L_t^{q_1\cap q_2}}<\infty.
\]

\vskip 0.2in

\noindent \textbf{Acknowledgement:} We thank F. Buseghin, C. Epstein, F. Gesztesy, H. Kovar\'ik, R. Mazzeo, C. Muscalu, W. Schlag and N. Visciglia for stimulating discussions, and for some bibliographical material.   G.S. also acknowledges the support of the NSF grants  DMS-2052651 and 
DMS-2306378,  and of  the Simons Foundation
Collaboration Grant on Wave Turbulence.

\vskip 0.2in



\section{The unitary group $e^{i t \Ba}$ and the Hankel transform}\label{S:solve} 

In this section we provide for completeness a purely analytical derivation of the kernel \eqref{SaS0}. We also unravel the deep connection between the group $e^{i t \Ba}$ and the modified Hankel transform $\mathcal H_{\frac{a-1}2}$. We begin by providing a formal derivation of \eqref{SaS0} based on the heat equation. Assume that $a> - 1$, and consider the kernel of the heat semigroup $e^{-t \Ba}$, subject to the Neumann condition \eqref{neu}
\begin{align}\label{heatsg}
p_a(x,y,t) & = (2t)^{-\frac{a+1}2} \left(\frac{xy}{2t}\right)^{\frac{1-a}2} I_{\frac{a-1}2}\left(\frac{xy}{2t}\right) e^{-\frac{x^2+y^2}{4t}},\ \ \ \ x, y, t>0,
\end{align}
where $I_\nu(z)$ denotes the modified Bessel function of the first kind, see \cite[(5.7.1), p.108]{Le}. 
The kernel \eqref{heatsg} is well-known, see \cite[p. 60]{IM}\footnote{Their normalization of the operator \eqref{Ba} differs from ours by a factor of $1/2$. Also, these authors assume that $a = d-1$ with $d\in \N$ since they base their probabilistic derivation on the well-known identity valid for any $z>0$ and $\omega \in \mathbb S^{d-1}$,
\[
\int_{\mathbb S^{d-1}} \exp \left\{z\sa\omega,x\da\right\} d\sigma(x) =  \G(d/2) \left(\frac z2\right)^{1-\frac d2} I_{\frac d2 - 1}(z).
\]}
\cite[Ex.6, p. 238]{KT}, and also \cite[Sec. 2]{BHNV}. For an analytical derivation that covers the range $a>-1$, see \cite[Prop. 22.3]{Gft}. A different analytical proof of \eqref{heatsg} can be obtained by appealing to the construction of the heat kernel $e^{-t \mathscr L_b}$ in \cite{EM} of the Kimura model, see \cite[Prop. 5.1]{Gams} and Section \ref{EM} for more details. Denote now by $J_\nu(z)$ the Bessel function of the first kind and order $\nu$, see \eqref{besseries}. If with $\nu = \frac{a-1}2\in (-1,\infty)$, we use the well-known relation (see \cite[(5.7.4), p.109]{Le}),
\[
I_\nu(z) = e^{-i\frac{\pi \nu}2} J_\nu(e^{i\frac{\pi}2} z),\ \ \ \ \ \ -\pi<\arg z< \frac{\pi}2,
\]
which gives
\[
I_{\frac{a-1}2}(\frac{xy}{2 i t}) = e^{i\frac{\pi(1-a)}4} J_{\frac{a-1}2}(\frac{xy}{2t}),
\]
we obtain from \eqref{heatsg} 
\[
p_a(x,y,it) = \frac{e^{-i \frac{(a+1) \pi}{4}}}{(2 t)^{\frac{a+1}2}}  \left(\frac{xy}{2 t}\right)^{\frac{1-a}2}J_{\frac{a-1}2}(\frac{xy}{2 t}) e^{i \frac{x^2+y^2}{4t}},
\]
which is the kernel $S_a(x,y,t)$ in \eqref{SaS0} for $t>0$. The expression of the kernel for $t<0$ is then found by invoking \cite[Lemma 4.1]{HMMS}.

Despite its obvious appeal, this quick presentation does not shed any light on the role of the Hankel transform in the analysis of the differential operator $\Ba$, as such connection is hidden in the derivation of \eqref{heatsg}. 
For this reason, in what follows we provide a self-contained proof of \eqref{SaS0}. In what follows we denote
\begin{equation}\label{gnu}
G_\nu(z) = z^{-\nu} J_\nu(z).
\end{equation}
Given $\nu\in \C$ with $\Re \nu>-1$, we define the \emph{modified Hankel transform} of a function $\vf:\R^+\to \overline \C$ as
\begin{equation}\label{ht1}
\mathcal H_\nu(\vf)(x) = \int_0^\infty \vf(y) G_\nu(xy) y^{2\nu+1} dy = x^{-\nu}\int_0^\infty \vf(y) J_\nu(xy) y^{\nu+1} dy,\ \ \ \ \Re \nu>-1,
\end{equation}
see \cite{MS}. We emphasize that \eqref{ht1} is related to the classical Hankel transform $H_\nu$ in 
\cite{Ha} by the formula
\begin{equation}\label{ht11}
\mathcal H_\nu(\vf)(x) = x^{-(\nu+\frac 12)} H_\nu((\cdot)^{\nu + \frac 12} \vf)(x),
\end{equation}
see also \cite{Wa, Ze, Ze2}.
The identity \eqref{ht11} is intimately connected to the transformations \eqref{uv}, \eqref{remove} below that link the Bessel operator $\Ba$ (for this, one must take $a = 2\nu+1$) to its quantum mechanics alter ego \eqref{Hm}. The Hankel transform \eqref{ht1} can be thought of as a fractal version of the Fourier transform of a function $f:\Rd\to \overline \C$, $d\in \N$,
\[
\hat f(\xi) = \int_{\Rd} e^{-2\pi i\sa\xi,w\da} f(w) dw.
\]
If $f(w)=F(|w|)$ is spherically symmetric, and if $z = |w|$, then Bochner's formula in \cite[Theor.40, p.69]{BC} gives
\begin{equation}\label{FB}
\hat f(\xi)=2\pi|\xi|^{-\frac{d}2 +1}\int^\infty_0  F(z)J_{\frac{d}2-1}
(2\pi|\xi|z) z^{\frac{d}2} dz.
\end{equation}
Comparing \eqref{ht1} with \eqref{FB} we see that   
\begin{equation}\label{hf}
\hat f(\xi) = (2\pi)^{\frac{a+1}2} \mathcal H_{\frac{a-1}2}(F)(2\pi |\xi|),\ \ \ \ \ a = d-1.
\end{equation}
In view of the equations \eqref{delta} and \eqref{hf}, it is clear that the modified Hankel transform \eqref{ht1} plays for the Bessel operator $\Ba$ the same role of the Fourier transform for the Laplacian.

 A fundamental property of $\mathcal H_\nu$ is the Plancherel type identity 
\begin{align}\label{planrad}
& \int_0^\infty |\vf(x)|^2 x^{2\nu+1} dx = \int_0^\infty |\mathcal H_{\nu}(\vf)(x)|^2 x^{2\nu+1} dx,\ \ \ \ \ \Re\nu>-1,
\end{align}
valid for any function $\vf\in \sn$, see for instance \cite[Lemma 2.7 \& p.125]{BS}.
From \eqref{planrad} and polarization, we obtain the following Parseval identity for $\vf, \psi\in \sn$
\begin{equation}\label{pol}
\int_0^\infty \vf(x) \overline{\psi(x)} x^{2\nu+1} dx = \int_0^\infty \mathcal H_{\nu}(\vf)(x)\overline{\mathcal H_{\nu}(\psi)(x)} x^{2\nu+1} dx.
\end{equation}
In his work \cite{Ha}, Hankel proved the inversion formula for his transform. For the modified Hankel transform \eqref{ht1} there exists a related inversion formula, see \cite{Sch, Sch2} for $\nu>-1/2$. A nice alternative proof is contained in \cite[Theor.1.1]{Co}. Although the authors state it for the range $\Re\nu > -1/2$, a careful examination shows that their proof is actually valid in the range $\Re\nu>-1$: we provide some details  in the proof of the next theorem for completeness. We mention that a different proof of Theorem \ref{T:hankel} can be found in \cite[Lemma 2.7 \& Sec. 4]{BS}.

\begin{theorem}\label{T:hankel}
Let $\Re\nu>-1$. Then for any $\vf\in \sn$ 
one has for $x> 0$
\[
\vf(x) = \mathcal H_{\nu}(\mathcal H_{\nu}(\vf))(x).
\]
\end{theorem}

\begin{proof}
In \cite{Co} the authors observe that, if $f(y) = e^{-\frac{y^2}2}$, then for $\Re\nu>-1/2$ the Hankel transform of $f$ is the function itself. In fact, such result follows from Hankel's generalization of Weber's integral which is valid for $\Re\nu>-1$, and for $|\arg p|< \frac{\pi}4$, see \cite[(10), p.469]{Ha} and also \cite[(4), p.394]{Wa}:
\begin{equation}\label{han}
\int_0^\infty e^{-p^2 y^2} J_\nu(xy) y^{\nu+1} dy = \frac{x^\nu}{(2 p^2)^{\nu+1}} \exp\left(-\frac{x^2}{4p^2}\right).
\end{equation} 
Comparing with \eqref{ht1}, the choice $p^2 = 1/2$ in formula \eqref{han} gives in particular
\begin{equation}\label{hanexp}
\mathcal H_\nu(e^{-(\cdot)^2/2})(x) =  e^{-x^2/2},\ \ \ \ \ \ \ \Re\nu>-1.
\end{equation}
With \eqref{hanexp} in hands, one can now repeat the proof in \cite{Co} to reach the desired conclusion.

\end{proof}

\vskip 0.2in

 The following result plays a critical role in the analysis of the Bessel operator \eqref{Ba}. It shows that, for the modified Hankel transform \eqref{ht1}, the degenerate operator $\Ba$ actually looks like a constant coefficient second order operator. For its proof we refer the reader to e.g. \cite[Sec. 22, Lemmas 22.1, 22.2]{Gft}.

\begin{theorem}\label{T:sq}
Let $\Re\nu>-1$, and $\vf\in \sn$. Then
\[
\mathcal H_{\nu}(\mathcal B_{2\nu+1} \vf)(x)= - x^2\  \mathcal H_{\nu}(\vf)(x).
\]
\end{theorem}

A formal way of ``seeing" Theorem \ref{T:sq} is to take $a = 2\nu+1$, so that $\nu = \frac{a-1}2$. Keeping \eqref{delta} in mind, if $f(w) = F(x)$, with $x = |w|$, and $F\in C_0^2(\overline{\R^+})$, we have
\[
\Delta_w f(w) = \Ba F(x).
\]
We thus formally obtain
\[
\widehat{\Ba F}(x) = \widehat{\Delta_w f}(w) = - 4\pi^2 |w|^2 \hat f(w).
\]
Applying \eqref{hf} to both sides of this identity, we find
\[
\mathcal H_{\frac{a-1}2}(\Ba F)(2\pi x) = - 4\pi^2 x^2 \mathcal H_{\frac{a-1}2}(F)(2\pi x) = - (2\pi x)^2 \mathcal H_{\frac{a-1}2}(F)(2\pi x),
\]
which provides the intuition behind Theorem \ref{T:sq}. We will also need the following result, see formulas \cite[6.729 (1 \& 2), p.758]{GR}. For our purposes it is critical that \eqref{b2} is valid for the range $\Re\nu>-1$.
\begin{lemma}\label{L:beauty}
Let $t, x, y>0$, and let $\Re\nu>-2$. Then
\begin{align}\label{b1}
& \int_0^\infty z \sin(t z^2) J_\nu(yz) J_\nu(xz) dz = \frac{1}{2t} \cos\left(\frac{x^2+y^2}{4t} - \frac{\nu \pi}{2}\right) J_\nu(\frac{xy}{2t}).
\end{align}
If instead $t, x, y>0$, and $\Re\nu>-1$, then
\begin{align}\label{b2}
& \int_0^\infty z \cos(t z^2) J_\nu(yz) J_\nu(xz) dz = \frac{1}{2t} \sin\left(\frac{x^2+y^2}{4t} - \frac{\nu \pi}{2}\right) J_\nu(\frac{xy}{2t}).
\end{align}
\end{lemma}

  With Theorem \ref{T:hankel}, Theorem \ref{T:sq} and Lemma \ref{L:beauty} in hands, we can finally turn to the

\begin{proof}[Proof of Proposition \ref{P:rep}]
To solve the Cauchy problem \eqref{cp0}, we first assume that $F\equiv 0$ and, with $\nu = \frac{a-1}2\in (-1,\infty)$, we apply the Hankel transform $\mathcal H_\nu$ in the variable $x>0$. Using Theorem \ref{T:sq}, we find that \eqref{cp0} is converted into the following one
\begin{equation}\label{cpB2}
\begin{cases}
\frac{\p\mathcal H_\nu(u)}{\p t}(x,t) + i x^2 \mathcal H_\nu(u)(x,t) = 0,\ \ \ \ \ \ \ \ \ \ \ x>0, t\in \R,
\\
\mathcal H_\nu(u)(x,0) = \mathcal H_\nu(\vf)(x).
\end{cases}
\end{equation}
For any fixed $x>0$,  the problem \eqref{cpB2} is solved by
\begin{equation}\label{y}
\mathcal H_\nu(u)(x,t) = e^{-i t x^2} \mathcal H_\nu(\vf)(x).
\end{equation}
Applying Theorem \ref{T:hankel} to \eqref{y}, we obtain
\begin{equation}\label{u}
u(x,t) = \mathcal H_\nu(e^{-i t (\cdot)^2} \mathcal H_\nu(\vf))(x).
\end{equation}
Using \eqref{ht1} we find
\begin{equation}\label{solB}
\mathcal H_\nu(e^{-i t (\cdot)^2} \mathcal H_\nu(\vf))(x) = 
x^{-\nu} \int_0^\infty \vf(y) y^{\nu+1} \left(\int_0^\infty z e^{-i t z^2} J_\nu(yz) J_\nu(xz) dz\right) dy
\end{equation}
To compute the inner integral in \eqref{solB} we suppose first that $t>0$, and write
\begin{align*}
& \int_0^\infty z e^{-i t z^2} J_\nu(yz) J_\nu(xz) dz = \int_0^\infty z \cos(t z^2) J_\nu(yz) J_\nu(xz) dz
 - i \int_0^\infty z \sin(t z^2) J_\nu(yz) J_\nu(xz) dz.
\end{align*}
The integrals in the right-hand side can be computed using Lemma \ref{L:beauty}. 
Since $\nu > -1$, and $\sin \theta - i \cos\theta = - i [\cos\theta + i \sin \theta]$, we obtain 
\begin{align}\label{inint}
& \int_0^\infty z e^{-i t z^2} J_\nu(yz) J_\nu(xz) dz = - \frac{i}{2t} J_\nu(\frac{xy}{2t}) e^{i \left(\frac{x^2+y^2}{4t} - \frac{\nu \pi}{2}\right)}
\\
& = - \frac{i}{2|t|} J_\nu(\frac{xy}{2|t|}) e^{i \frac{x^2+y^2}{4|t|}} e^{- i \frac{\nu \pi}{2}} .
\notag
\end{align}
If instead $t<0$, then \eqref{u} becomes
\begin{equation}\label{uu}
u(x,t) = \mathcal H_\nu(e^{i |t| (\cdot)^2} \mathcal H_\nu(\vf))(x),
\end{equation}
and we are thus led to compute the integral
\begin{align*}
& \int_0^\infty z e^{i |t| z^2} J_\nu(yz) J_\nu(xz) dz = \int_0^\infty z \cos(|t| z^2) J_\nu(yz) J_\nu(xz) dz
 + i \int_0^\infty z \sin(|t| z^2) J_\nu(yz) J_\nu(xz) dz.
\end{align*}
Lemma \ref{L:beauty} gives in this case
\begin{align}\label{uubar}
& \int_0^\infty z e^{i |t| z^2} J_\nu(yz) J_\nu(xz) dz =  \frac{1}{2|t|} \sin\left(\frac{x^2+y^2}{4|t|} - \frac{\nu \pi}{2}\right) J_\nu(\frac{xy}{2|t|})
\\
& + i \frac{1}{2|t|} \cos\left(\frac{x^2+y^2}{4|t|} - \frac{\nu \pi}{2}\right) J_\nu(\frac{xy}{2|t|}) =   \frac{i}{2|t|} J_\nu(\frac{xy}{2|t|}) e^{- i \left(\frac{x^2+y^2}{4|t|} - \frac{\nu \pi}{2}\right)}.
\notag
\end{align}
Recalling that $a = 2\nu+1$, so that $\nu = \frac{a-1}2$ and $\nu+1 = \frac{a+1}2$, and substituting \eqref{inint}, \eqref{uubar} in \eqref{solB}, \eqref{u}, we finally reach the conclusion that, when $F\equiv 0$, the problem \eqref{cp0} is uniquely solved by the function 
\begin{equation}\label{solBB}
u(x,t) = S_a(t)\vf(x) \overset{def}{=} \int_0^\infty S_a(x,y,t) \vf(y)  d\omega_a(y),\ \ \ \ \ x\ge 0,\ \  t\in\R,
\end{equation}
where $S_a(x,y,t)$ is given by \eqref{SaS0}. The non-homogeneous case $F\not= 0$ follows now from \eqref{solBB} via Duhamel's principle.

\end{proof}

Comparing the expression \eqref{SaS0} with the definition \eqref{ht1} of the Hankel transform, we immediately obtain the following result.

\begin{proposition}\label{P:cpH}
Let $a>-1$. Given $\vf\in \sn$, denote by $S_a(t)\vf(x)$ the unique solution to the homogeneous Cauchy problem \eqref{cp0}. Then for any $t>0$ we have
\[
S_a(t)\vf(x) = e^{-i \frac{(a+1) \pi}{4}} e^{i \frac{x^2}{4t}} \mathcal H_{\frac{a-1}2}(e^{i \frac{(\cdot)^2}{4t}} \vf)(\frac{x}{2t}).
\]
\end{proposition}
Proposition \ref{P:cpH} shows that, up to rescaling and multiplication by the unimodular function $e^{i \frac{x^2}{4t}}$, the solution to the Cauchy problem \eqref{cp0} with initial datum $\vf$ \emph{is} the Hankel transform of $\vf$. This is akin to the well-known connection between Fourier transform and the solution to the homogeneous Cauchy problem for the Schr\"odinger equation.  We next prove property (iii) following Proposition \ref{P:rep}. 

\begin{proposition}\label{P:1}
Let $-1<a<2$. For every $x\in \R$ and $t\in \R$ one has
\[
\int_0^\infty S_a(x,y,t) d\omega_a(y) = 1.
\]
\end{proposition}

\begin{proof}
We only consider the case $t>0$. We have from \eqref{SaS0}
\begin{align*}
& \int_0^\infty S_a(x,y,t) d\omega_a(y) = \frac{e^{-i \frac{(a+1) \pi}{4}}}{(2t)^{\frac{a+1}2}} e^{i \frac{x^2}{4t}} \int_0^\infty \left(\frac{xy}{2t}\right)^{\frac{1-a}2}J_{\frac{a-1}2}(\frac{xy}{2t}) e^{i \frac{y^2}{4t}} y^{a+1} \frac{dy}y
\\
& = 2^{\frac{a+1}2} e^{-i \frac{(a+1) \pi}{4}} e^{i \frac{x^2}{4t}} \left( \frac{x}{\sqrt t} \right)^{\frac{1-a}2} \int_0^\infty z^{\frac{a+1}2}J_{\frac{a-1}2}(\frac{x}{\sqrt t} z) e^{i z^2}  dz.
\end{align*}
The proof will be finished if we show that
\begin{equation}\label{2}
\int_0^\infty z^{\frac{a+1}2}J_{\frac{a-1}2}(\frac{x}{\sqrt t} z) e^{i z^2}  dz = 2^{-\frac{a+1}2} e^{i \frac{(a+1) \pi}{4}} e^{-i \frac{x^2}{4t}} \left(\frac{x}{\sqrt t} \right)^{\frac{a-1}2}.
\end{equation}
The identity \eqref{2} follows by applying with $\nu = \frac{a-1}2$, $\alpha = 1$ and $b = t^{-1/2} x$, the following equations, which can be found in e.g. \cite[6.728.5 \& 6.728.6, p.739]{GR}: for $\alpha, b>0$ one has
\[
\int_0^\infty z^{\nu+1} J_\nu(bz) \sin(\alpha z^2) dz = \frac{b^\nu}{2^{\nu+1} \alpha^{\nu+1}} \cos\left(\frac{b^2}{4\alpha} - \frac{\pi \nu}2\right),\ \ \ \ -2<\Re \nu < \frac 12,
\]
and
\[
\int_0^\infty z^{\nu+1} J_\nu(bz) \cos(a z^2) dz = \frac{b^\nu}{2^{\nu+1} a^{\nu+1}} \sin\left(\frac{b^2}{4\alpha} - \frac{\pi \nu}2\right),\ \ \ \ -1<\Re \nu < \frac 12,
\]

\end{proof} 

\vskip 0.2in

We close this section by noting that the case $a=0$ of \eqref{SaS0} is special. In fact in the next proposition we use  formula \eqref{SaS0} to recover the well known expression of the solution to the  linear Schr\"odinger equation via the Schr\"odinger group. 

\begin{proposition}\label{P:classic}
Let $a = 0$, and $\vf\in \mathscr S(\R)$ and even. Then the solution of \eqref{cp0}, with $F\equiv 0$, is given by 
\begin{equation}\label{classical}
u(x,t) = S_0(t) \vf(x) =  (4\pi i t)^{-1/2} \int_{\R} e^{i \frac{(x-y)^2}{4t}} \vf(y) dy,
\end{equation}
the unique (even in $x$) solution of the Cauchy problem on the line $\R$
\[
\p_t u - i \p_{xx} u = 0, \ \ \ \ \ \ u(x,0) = \vf(x).
\]
\end{proposition}

\begin{proof}
 To recognise \eqref{classical}, we set $a=0$ in \eqref{SaS0}, obtaining for $x\ge 0$ and $t\in \R$,
\begin{align*}
& S_0(t) \vf(x) = \frac{e^{- i \frac{\pi}{4}}}{(2t)^{\frac{1}2}}  \int_0^\infty \left(\frac{xy}{2t}\right)^{\frac{1}2}J_{-\frac{1}2}(\frac{xy}{2t}) e^{i \frac{x^2+y^2}{4t}} \vf(y) dy.
\end{align*}
Keeping in mind that
\begin{equation}\label{jmeno12}
J_{-\frac12}(z)=\sqrt{\frac2{\pi z}}\cos z, \end{equation}
see e.g. \cite[(5.8.2), p.111]{Le},
we find
\begin{align*}
& S_0(t) \vf(x) = \frac{e^{- i \frac{\pi}{4}}}{(2t)^{\frac{1}2}}  \int_0^\infty \left(\frac{xy}{2t}\right)^{\frac{1}2}\sqrt{\frac2{\pi (\frac{xy}{2t})}}\cos (\frac{xy}{2t}) e^{i \frac{x^2+y^2}{4t}} \vf(y) dy
\\
& = \sqrt{\frac2{\pi}}\frac{1}{(2i t)^{\frac{1}2}}  \int_0^\infty \cos (\frac{xy}{2t}) e^{i \frac{x^2+y^2}{4t}} \vf(y) dy.
\end{align*}
We now write
\begin{align*}
& \cos (\frac{xy}{2t}) e^{i \frac{x^2+y^2}{4t}} = \frac 12 \left[e^{i \frac{2xy}{4t}} + e^{-i \frac{2xy}{4t}}\right]e^{i \frac{x^2+y^2}{4t}} 
\\
& = \frac 12 \left[e^{i \frac{(x-y)^2}{4t}} + e^{i \frac{(x+y)^2}{4t}}\right],
\end{align*}
obtaining
\begin{align*}
& S_0(t) \vf(x) = \frac{1}{(4\pi i t)^{\frac{1}2}} \left[\int_0^\infty e^{i \frac{(x-y)^2}{4t}}  \vf(y) dy + \int_0^\infty e^{i \frac{(x+y)^2}{4t}}  \vf(y) dy\right]
\\
& =  \frac{1}{(4\pi i t)^{\frac{1}2}} \left[\int_0^\infty e^{i \frac{(x-y)^2}{4t}}  \vf(y) dy + \int_{-\infty}^0 e^{i \frac{(x-y)^2}{4t}}  \vf(-y) dy\right]
\\
& = (4\pi i t)^{-1/2} \int_{\R} e^{i \frac{(x-y)^2}{4t}} \vf(y) dy,
\end{align*}
the last equality being justified by the fact that $\vf$ is even on $\R$.

\end{proof}

We now consider the problem
\begin{equation}\label{cp1}
\begin{cases}
\p_t v - i \Ba v = G(x,t),\ \ \ \ \ x\in \R^+,\ t\in \R,\ \ a >-1, 
\\
\underset{x\to 0^+}{\lim} x^a \p_x v(x,t) = \Phi(t),\ \ \ \ t\in \R,
\\
v(x,0) = \vf(x),
\end{cases}
\end{equation}
for which we assume that $\Phi(0) = 0$. To solve \eqref{cp1}, for a given solution $v$, we consider the function
\begin{equation}\label{v}
u(x,t) = v(x,t) - \frac{x^{1-a}}{1-a} \Phi(t).
\end{equation}
One easily checks that the function \eqref{v} satisfies the problem \eqref{cp0} with 
\[
F(x,t) = G(x,t) - \frac{x^{1-a}}{1-a} \Phi'(t).
\]
If we assume that $\vf\in \sn$ and that $G$ is such that $F\in \snn$, then applying Proposition \ref{P:rep} to $u$, we can represent $v$ by the formula
\begin{align}\label{kernel2}
v(x,t) & = \frac{x^{1-a}}{1-a} \Phi(t) +  \int_0^\infty S_a(x,y,t) \vf(y) d\omega_a(y) \\
& + \int_0^t \int_0^\infty S_a(x,y,t-s) F(y,s) d\omega_a(y) ds.
\notag
\end{align}

\section{Dispersive estimates}\label{S:distri}

The main result of this section is the dispersive estimate in Proposition \ref{P:dis}. The latter shows that the regimes $-1<a<0$ and $a\ge 0$ are dramatically different, an aspect that will be reflected in the Strichartz estimates in Section \ref{S:stri}. Using the kernel \eqref{SaS0}, for every $t\in \R$ we define a linear operator $S_a(t)$ according to the rule 
\begin{equation}\label{Tstar}
S_a(t)\vf(x) =  
\int_0^\infty S_a(x,y,t) \vf(y) d\omega_a(y),\ \ \ \ \ \vf\in \sn.
\end{equation}
Our first result is the following.

\begin{proposition}\label{P:uni}
For every $t\in \R$, the operator $S_a(t): L^2_a\to L^2_a$ unitarily, in the sense that for any $t\in \R$
\[
||S_a(t)\vf||_{L^2_a} = ||\vf||_{L^2_a}.
\]
\end{proposition}

\begin{proof}
We note that, with $a = 2\nu+1$, formula \eqref{y} can be expressed as follows 
\begin{equation}\label{hank}
\mathcal H_{\frac{a-1}2}(S_a(t) \vf)(x) = e^{-i t x^2} \mathcal H_{\frac{a-1}2}(\vf)(x).
\end{equation}
Using Plancherel formula \eqref{planrad}, we thus find for every $t\in \R$ 
\begin{align}\label{uni}
& \int_0^\infty |S_a(t)\vf(x)|^2 d\omega_a(x) = \int_0^\infty |\mathcal H_{\frac{a-1}2}(S_a(t)\vf)(x)|^2 d\omega_a(x)
\\
& = \int_0^\infty |\mathcal H_{\frac{a-1}2}(\vf)(x)|^2 d\omega_a(x) = \int_0^\infty |\vf(x)|^2 d\omega_a(x),
\notag
\end{align}
where in the second equality we have used \eqref{hank}, and in the third \eqref{planrad} again. The equation \eqref{uni} proves the desired conclusion.

\end{proof}

We next want to show that $S_a(t): L^1_a\to L^\infty$, with a certain bound. We will need the following lemma.

\begin{lemma}\label{L:dis}
Let $a\ge 0$. There exists $C(a)>0$ such that 
\begin{equation}\label{goodone}
|S_a(x,y,t)| \le C(a)\ |t|^{-\frac{a+1}2},\ \ \ \ x, y\ge 0,\ t\in \R\setminus\{0\}. 
\end{equation}
If instead $-1<a<0$, then we have the following bound
\begin{equation}\label{notsogoodone}
|S_a(x,y,t)| \le \begin{cases}
C(a)\ |t|^{-\frac{a+1}2},\ \ \ \ \ \ \ \ \ \ \frac{xy}{2|t|}\le 1,
\\
C(a)\ |t|^{-\frac{1}2} (xy)^{-\frac a2},\ \ \ \ \frac{xy}{2|t|}\ge 1\ \emph{and}\ xy\ge 1,
\\
C(a)\ |t|^{-\frac{1}2},\ \ \ \ \ \ \ \ \ \ \ \ \ \frac{xy}{2|t|}\ge 1\ \emph{and}\ xy\le 1.
\end{cases} 
\end{equation}
\end{lemma}

\begin{proof}
From the expression \eqref{SaS0}, we find
\begin{align}\label{Sabound}
|S_a(x,y,t)| & \le \frac{1}{(2|t|)^{\frac{a+1}2}}  \left(\frac{xy}{2|t|}\right)^{\frac{1-a}2} \left|J_{-\frac{1-a}2}(\frac{xy}{2|t|})\right|.
\end{align}
In view of \eqref{Js}, when $0\le \frac{xy}{2|t|}\le 1$, we have for any $a>-1$
\[
\left(\frac{xy}{2|t|}\right)^{\frac{1-a}2} \left|J_{-\frac{1-a}2}(\frac{xy}{2|t|})\right|\le C(a),
\]
whereas when $\frac{xy}{2|t|}\ge 1$ we use the asymptotic behavior \eqref{jnuinfty}  in the appendix to infer that 
\[
\left|J_{-\frac{1-a}2}(\frac{xy}{2|t|})\right|\le C(a) \left(\frac{xy}{2|t|}\right)^{-\frac 12}.
\] 
This estimate gives
\[
\left(\frac{xy}{2|t|}\right)^{\frac{1-a}2} \left|J_{-\frac{1-a}2}(\frac{xy}{2|t|})\right|\le C(a) \left(\frac{xy}{2|t|}\right)^{-\frac{a}2}.
\]
When $a\ge 0$, it is clear that $\left(\frac{xy}{2|t|}\right)^{-\frac{a}2}\le 1$, when $\frac{xy}{2|t|}\ge 1$, and we reach the desired conclusion \eqref{goodone}. If instead $-1<a<0$ the bound \eqref{notsogoodone} follows from \eqref{Sabound} and the above asymptotic estimates.

\end{proof}

\begin{definition}\label{D:mn}
For $-1<a<0$, we define 
\[
k_1(x) = \min\{1,x^{\frac a2}\},\ \ \ \ \ \text{and}\ \ \ \ \ \ u_1(x) = \max\{x^{\frac a2},x^a\} = k_1(x)^{-1} x^a.
\]  
\end{definition}
The weight function $k_1(x)$ will play a critical role in our Strichartz estimates in the regime $-1<a<0$. We have the following basic result.

\begin{proposition}[Dispersive estimate]\label{P:dis}
Suppose that $a\ge 0$. There exists $C(a)>0$ such that for any $t\in \R\setminus\{0\}$ we have 
\begin{equation}\label{good}
||S_a(t) \vf||_{L^\infty(\R^+)}\le C(a) |t|^{-\frac{a+1}2} ||\vf||_{L^1_a}.
\end{equation}
If instead $-1<a<0$, then there exists $C(a)>0$ such that for any $t\in \R\setminus\{0\}$ we have 
\begin{equation}\label{weight}
||S_a(t)\vf\ k_1||_{L^\infty(\R^+)} \le C(a)\ \left\{|t|^{-\frac{a+1}2} + |t|^{-\frac{1}2}\right\} ||\vf\ u_1||_{L^1(\R^+)}.
\end{equation} 
\end{proposition}

\begin{proof}
For the regime $a\ge 0$, the estimate \eqref{good} is an immediate consequence of  \eqref{goodone} in Lemma \ref{L:dis}. When $-1<a<0$, the estimate \eqref{good} fails. To prove \eqref{weight}, 
given points $x, |t|>0$ we consider the following subsets of the half-line
\[
A_{x,t} = \{y>0\mid \frac{xy}{2|t|} \le1\},\ \ \ B_{x,t} = \{y>0\mid \frac{xy}{2|t|} >1, xy\le 1\}, \ \ \ C_{x,t} = \{y>0\mid \frac{xy}{2|t|} >1, xy> 1\}.
\]
Let $\vf\in \sn$, then \eqref{notsogoodone} in Lemma \ref{L:dis} gives
\begin{align*}
& |S_a(t)\vf(x)| \le \int_{A_{x,t}} |S_a(x,y,t)| |\vf(y)| y^a dy + \int_{B_{x,t}} |S_a(x,y,t)| |\vf(y)| y^a dy
\\
& + \int_{C_{x,t}} |S_a(x,y,t)| |\vf(y)| y^a dy
\\
& \le C(a)\ |t|^{-\frac{a+1}2} \int_{A_{x,t}} |\vf(y)| y^a dy + C(a)\ |t|^{-\frac{1}2} \int_{B_{x,t}}  |\vf(y)| y^a dy
\\
& + C(a)\ |t|^{-\frac{1}2} x^{-\frac a2} \int_{C_{x,t}} |\vf(y)| y^{\frac a2} dy
\\
& \le C(a)\ \left\{|t|^{-\frac{a+1}2} + |t|^{-\frac{1}2}\right\} \max\{1,x^{-\frac a2}\} \int_0^\infty |\vf(y)| \max\{y^a,y^{\frac a2}\} dy
\\
& = C(a)\ \left\{|t|^{-\frac{a+1}2} + |t|^{-\frac{1}2}\right\} \left(\min\{1,x^{\frac a2}\}\right)^{-1} \int_0^\infty |\vf(y)| \max\{y^a,y^{\frac a2}\} dy.
\end{align*}
Keeping Definition \ref{D:mn}
 in mind, this estimate implies \eqref{weight}.

\end{proof}

To establish the next result we will need E. Stein's generalization \cite[Theor. 2]{Eli} of the Riesz-Thorin interpolation theorem.  We recall it for the reader's convenience (for a version which allows for sublinear, instead of linear,  operators, see \cite{SW}). Let $u_i\ge 0$, $i=1,2$ be measurable functions on a measure space $M$, and $k_i\ge 0$, $i=1,2$ be measurable functions on another measure space $N$. Let $T$ be a linear operator mapping simple functions on  $M$ to measurable functions on $N$. Suppose that for $1\le p_i, q_1 \le \infty$, $i=1,2$, one has on simple functions $f$,
\[
||T f\ k_1||_{q_1} \le M_1 ||f u_1||_{p_1},\ \ \ \ \ ||T f\ k_2||_{q_2} \le M_2 ||f u_2||_{p_2}.
\]
Let $p, q$ be given by
\[
\frac 1p = \frac{1-\theta}{p_1} + \frac{\theta}{p_2},\ \ \ \ \ \ \frac 1q = \frac{1-\theta}{q_1} + \frac{\theta}{q_2},\ \ \ \ \ \ 0\le \theta\le 1,
\]
and define
\begin{equation}\label{ku}
k = k_1^{1-\theta} k_2^\theta,\ \ \ \ \ \ \ \ u = u_1^{1-\theta} u_2^\theta.
\end{equation}
Then $T$ may be uniquely extended to a linear map on functions $f$ such that $||f u||_p<\infty$, satisfying 
\begin{equation}\label{stein}
||Tf\ k||_q \le M_1^{1-\theta} M_2^\theta ||f u||_p.
\end{equation}
We also need the following consequence of Stein's theorem. Let $w\ge 0$ and consider the measure $wdx$ on $M = N$. Suppose that on simple functions $f$ one has
\[
||T f||_{L^\infty(dx)} \le M_1 ||f||_{L^1(wdx)},\ \ \ \ \ \ \ \ ||T f||_{L^2(wdx)}  \le M_2 ||f||_{L^2(wdx)}.
\]
The, for any $r\ge 2$ one has
\begin{equation}\label{stein2}
||T f||_{L^r(wdx)} \ \le M_1^{\frac{1}{r'}-\frac 1r} M_2^{\frac 2r} ||f||_{L^{r'}(wdx)}.
\end{equation} 
To prove \eqref{stein2} we apply Stein's theorem with 
\[
k_1 = 1,\ \ \ \ u_1 = w,\ \ \ \ \ k_2 = u_2 = w^{1/2}.
\]
For $r\ge 2$ we write
\[
\frac{1}{r'} = \frac{1-\theta}1 + \frac{\theta}2 = 1 - \frac{\theta}2\ \Longrightarrow\ \frac{\theta}2 = \frac 1r.
\] 
For such value of $\theta$, we have 
\[
k = w^{\frac{\theta}2} = w^{\frac 1r},\ \ \ \ \ \ u = w^{1-\theta} w^{\frac{\theta}2} = w^{1-\frac{\theta}2} = w^{\frac 1{r'}}.
\]
Then the conclusion \eqref{stein2} follows from \eqref{stein}. With these results in hands, we now prove the following crucial dispersive estimate.

\begin{proposition}\label{P:disp}
Let $r\ge 2$. For any $a\ge 0$ there exists a constant $C(a,r)>0$ such that for any $\vf\in \sn$ one has
\begin{equation}\label{eli1}
||S_a(t)\vf||_{L^r_a} \le C(a,r) |t|^{-(a+1)(\frac 12 - \frac 1r)} ||\vf||_{L^{r'}_a}.
\end{equation}
If instead $-1<a<0$, we have for some $C(a,r)>0$
\begin{equation}\label{eli2}
||S_a(t)\vf k_1^{1-\frac 2r} ||_{L^r_a} \le C(a,r) \left\{|t|^{-(a+1)(\frac 12 - \frac 1r)} + |t|^{-(\frac{1}2-\frac 1r)}\right\} ||\vf k_1^{1-\frac{2}{r'}}||_{L^{r'}_a}.
\end{equation}
\end{proposition}

\begin{proof}
The proof of \eqref{eli1} is an immediate consequence of \eqref{uni},  \eqref{good} and Stein's \eqref{stein2}, which we apply with $w(x) = x^a$. Since
\[
M_1 = C(a) |t|^{-\frac{a+1}2},\ \ \ \ M_2 = 1,
\]
and $\theta = \frac 2r$, we have 
\[
M_1^{1-\theta} = C(a,r) |t|^{-(a+1)(\frac 12 - \frac 1r)}.
\]
To prove \eqref{eli2} instead, we apply Stein's theorem \eqref{stein} with  $k_1(x)$ and $u_1(x)$ as in Definition \ref{D:mn}, and with  
\[
u_2(x) = k_2(x) = x^{\frac a2}.
\]
This time we have
\[
M_1(t) = C(a)\ \left\{|t|^{-\frac{a+1}2} + |t|^{-\frac{1}2}\right\},\ \ \ \ \ M_2(t) \equiv 1,
\]
see \eqref{weight} and \eqref{uni}. For every $t\not= 0$, we thus find from \eqref{stein} for every $\vf\in \sn$
\begin{equation}\label{eli3}
||S_a(t)\vf\ k||_{L^r(\R^+)} \le M_1(t)^{1-\theta} ||\vf\ u||_{L^{r'}(\R^+)},
\end{equation}
where $k$ and $u$ are as in \eqref{ku}.
Now we compute the measures $k(x)^r dx$ and $u(x)^{r'} dx$ on $\R^+$. Since $\frac{\theta}2 = \frac 1r$, we have $(1-\theta)r = r-2$, and therefore
\begin{equation}\label{kr}
k(x)^r  = k_1(x)^{(1-\theta)r} x^{r \frac a2 \theta} = k_1(x)^{r-2} x^{a}.
\end{equation}
Next, we claim that
\begin{equation}\label{ur}
u(x)^{r'} = \frac{x^a}{k_1(x)^{2-r'}}.
\end{equation}
To verify \eqref{ur} we use Definition \ref{D:mn}, which gives
\begin{align*}
u(x)^{r'} & =  u_1(x)^{r'(1-\theta)} u_2(x)^{r'\theta} = k_1(x)^{-(1-\theta)r'} x^{a(1-\theta)r'} x^{\frac a2 \theta r'}
\\
& = k_1(x)^{-(1-\theta)r'} x^{a(1-\theta)r'} x^{\frac a2 \theta r'}. 
\end{align*}
Observing that $(1-\theta)r' = 2-r'$, and that $1-\theta + \frac{\theta}2 = \frac{1}{r'}$, we obtain \eqref{ur}.
Finally, we note that $1-\theta = 2(\frac 12 - \frac 1r)$ and therefore
\begin{equation}\label{mt}
M_1(t)^{1-\theta} \le C(a,r) \left\{|t|^{-(a+1)(\frac 12 - \frac 1r)} + |t|^{-(\frac{1}2-\frac 1r)}\right\}.
\end{equation}
Inserting \eqref{kr}, \eqref{ur} and \eqref{mt} in \eqref{eli3}, we reach the desired conclusion \eqref{eli2}.

\end{proof}


\section{Strichartz estimates}\label{S:stri}

In this section we prove Theorems \ref{T:GV} and \ref{T:restriction}. We will critically use Proposition \ref{P:disp} and ultimately base our approach on the well-known method of Ginibre and Velo in \cite{GV, GV2}, see also the book of Cazenave \cite{Caze}. For any $a>-1$ we define an operator $T_a^\star$, mapping functions $\vf:\R^+\to \C$ into functions on the half-plane $\R^+\times\R$:
\begin{equation}\label{Tstar1}
T_a^\star(\vf)(x,t) = S_a(t)\vf(x) = \int_0^\infty S_a(x,y,t) \vf(y) y^a dy.  
\end{equation}
Proposition \ref{P:uni} can be reformulated as follows.
 
\begin{proposition}\label{P:unistar}
The operator $T^\star_a: L^2_a\to L^\infty_t L^2_a$, and we have
\[
||T^\star_a\vf||_{L^\infty_t L^2_a} = ||\vf||_{L^2_a}.
\]
\end{proposition}

Next, we obtain a representation formula for the operator $T_a: L^1_t L^2_a\to L^2_a$, whose adjoint is $T_a^\star$.

\begin{lemma}\label{L:T}
For every function $F\in \snn$, we have
\begin{equation}\label{T}
T_a(F)(x) = \int_\R S_a(-t)(F(\cdot,t))(x) dt.
\end{equation}
\end{lemma}

\begin{proof}
Denote by $\sa\sa\cdot,\cdot \da\da$ the inner product in $L^2_t L^2_a$. For every $\vf\in \sn$, using twice \eqref{pol} and \eqref{hank}, along with \eqref{Tstar1}, we obtain 
\begin{align*}
\sa T_a(F),\vf\da & = \sa\sa F,T_a^\star(\vf)\da\da = \int_\R \int_0^\infty F(x,t) \overline{T_a^\star(\vf)(x,t)} d\omega_a(x) dt
\\
& = \int_\R  \int_0^\infty \mathcal H_{\frac{a-1}2}(F(\cdot,t))(x)\overline{\mathcal H_{\frac{a-1}2}(T_a^\star(\vf)(\cdot,t))(x)} d\omega_a(x) dt
\\
& =  \int_\R  \int_0^\infty e^{i t x^2} \mathcal H_{\frac{a-1}2}(F(\cdot,t))(x)\overline{\mathcal H_{\frac{a-1}2}(\vf)(x)} d\omega_a(x) dt
\\
& =  \int_\R  \int_0^\infty  \mathcal H_{\frac{a-1}2}(S_a(-t)F(\cdot,t))(x)\overline{\mathcal H_{\frac{a-1}2}(\vf)(x)} d\omega_a(x) dt
\\
& = \int_\R  \int_0^\infty S_a(-t)(F(\cdot,t))(x) \overline{\vf(x)} d\omega_a(x) dt
\\
& = \int_0^\infty\left(\int_\R S_a(-t)(F(\cdot,t))(x) dt\right) \overline{\vf(x)} d\omega_a(x)
\\
& = \big\sa \int_\R S_a(-t)(F(\cdot,t)) dt,\vf\big\da.
\end{align*}
This proves \eqref{T}.

\end{proof}

\begin{proposition}\label{P:Tunistar}
The operator $T_a: L^1_t L^2_a\to L^2_a$, and we have for any $F\in \snn$
\[
||T_a(F)||_{L^1_t L^2_a} \le  ||F||_{L^2_a}.
\]
\end{proposition}

\begin{proof}
From Lemma \ref{L:T} and \eqref{hank} it is clear that
\begin{equation}\label{TF}
||T_a(F)||_{L^2_a}  \le \int_\R ||S_a(-t)(F(\cdot,t))||_{L^2_a} dt
= \int_\R ||F(\cdot,t)||_{L^2_a} dt
 = ||F||_{L^1_t L^2_a}.
\end{equation}

\end{proof}

Suppose now that $F\in \snn$. Then, \eqref{Tstar1} and \eqref{T} in Lemma \ref{L:T} give
\begin{align}\label{stelle}
T_a^\star T_a(F)(x,t) & = S_a(t)(T_a(F))(x) = S_a(t) \int_\R S_a(-\tau)(F(\cdot,\tau))(x) d\tau
\\
& =  \int_\R S_a(t-\tau)(F(\cdot,\tau))(x) d\tau.
\notag
\end{align}

We next establish a basic estimate for the composition operator $T^\star T$.

\begin{theorem}\label{T:comp}
Let $a\ge 0$ and let the pair $(q,r)$ be $a$-admissible. There exists $C(a,r)>0$ such that for every $F\in \snn$ one has
\[
||T_a^\star T_a(F)||_{L^q_t L^r_a} \le C(a,r) ||F||_{L^{q'}_t L^{r'}_a}.
\]
\end{theorem}

\begin{proof}
The inequality \eqref{stelle} shows that, if $2\le r\le \infty$, then
\begin{align}\label{TstarT}
||T_a^\star T_a(F)(\cdot,t)||_{L_a^{r}} & \le \int_\R ||S_a(t-\tau)(F(\cdot,\tau))||_{L_a^{r}} d\tau
\\
& \le C(a,r) \int_\R |t-\tau|^{-(a+1)(\frac 12 - \frac 1r)} ||F(\cdot,\tau))||_{L_a^{r'}} d\tau, 
\notag
\end{align}
where in the last inequality we have used \eqref{eli1} in Proposition \ref{P:disp}. If we now consider the M. Riesz operator of fractional integration on $\R$
\begin{equation}\label{marcel}
I_\beta(h) = h\star |\cdot |^{-(1-\beta)},\ \ \ \ \ \ 0<\beta<1,   
\end{equation} 
then by the Hardy-Littlewood theorem (see \cite{St}) we know that, for any $1<p<1/\beta$, there exists $C(\beta,p)>0$ such that 
\begin{equation}\label{marcello}
\|I_\beta(h)||_{L^q(\R)} \le C(\beta,p) ||h||_{L^p(\R)},\ \ \ \ \ \ \ \forall h\in L^p(\R),
\end{equation}
provided that
\[
\frac 1p - \frac 1q = \beta.
\] 
In particular, if for a given $q\ge 1$ we know that $1<q'<1/\beta$, then we can take $p = q'$ in \eqref{marcello}, and obtain
\begin{equation}\label{marcellino}
\|I_\beta(h)||_{L^q(\R)} \le C(\beta,q) ||h||_{L^{q'}(\R)},\ \ \ \ \ \ \ \forall h\in L^{q'}(\R),
\end{equation}
provided that
\begin{equation}\label{a}
\frac 2q = 1-\beta.
\end{equation}
Notice that, since $0<1-\beta<1$, the equation \eqref{a} implies $2<q< \infty$. Furthermore, since the constraint $1<q'<1/\beta$ is equivalent to $0<\frac 1q<1-\beta$, and since $\frac 1q < \frac 2q = 1-\beta$ by \eqref{a}, it is clear that $0<\beta<1$ plus \eqref{a}, automatically imply $1<q'<1/\beta$.
Since from \eqref{TstarT} we see that we must take
\begin{equation}\label{hls}
1 - \beta = (a+1)(\frac 12 - \frac 1r),
\end{equation}
we infer that, in order to satisfy $0<\beta<1$,
we need to have 
\begin{equation}\label{hls2}
0< 1 - (a+1)(\frac 12 - \frac 1r)<1. 
\end{equation}
Since $a+1>0$,  to have $1 - (a+1)(\frac 12 - \frac 1r)<1$ we must have $r>2$. Secondly, for $0<1 - (a+1)(\frac 12 - \frac 1r)$
to hold, we must have
\[
(a+1)(\frac 12 - \frac 1r)<1\ \Longleftrightarrow\ \frac{a+1}2<1+\frac{a+1}r.
\]
This inequality is automatically true when $-1<a\le 1$. If $a>1$, we must impose the further constraint $r<\frac{2(a+1)}{a-1}$. Since $(q,r)$ is $a$-admissible, we know that the number $\beta$ defined by the equation \eqref{hls}, and the number $q'$, satisfy the constraints \eqref{hls2}. Therefore, since \eqref{TstarT} implies
\begin{equation}\label{TstarTT}
||T_a^\star T_a(F)(\cdot,t)||_{L_a^{r}} \le C(a,r) I_\beta(h)(t),
\end{equation}
with $h(t) = ||F(\cdot,t)||_{L^{r'}_a}$, we infer from \eqref{marcellino}
\[
||T_a^\star T_a(F)||_{L^q_t L_a^{r}} \le C^\star(a,r) ||I_\beta(h)||_{L^q_t} \le \overline C(a,r) ||h||_{L^{q'}_t} = \overline C(a,r) ||F||_{L^{q'}_t L^{r'}_a}.
\]

\end{proof}

Theorem \ref{T:comp} has the following two important consequences.

\begin{theorem}[Restriction estimates]\label{T:striT}
Under the assumptions of Theorem \ref{T:comp}, for every $F\in \snn$ we have
\[
||T_a(F)||_{L^2_a} \le C'(a,r) ||F||_{L^{q'}_t L^{r'}_a}.
\]
\end{theorem}

\begin{proof}
We have
\begin{align*}
||T_a(F)||^2_{L^2_a} & = \sa T_a(F),T_a(F)\da_{L^2_a} = \sa\sa T_a^\star(T_a(F)),F\da\da_{L^2_t L^2_a} 
\\
& = \int_\R \int_0^\infty \overline{F(x,t)} T^\star_a(T_a(F))(x,t) d\omega_a(x) dt
\\
& \le \int_\R\left(\int_0^\infty |T^\star_a(T_a(F))(x,t)|^r d\omega_a(x)\right)^{\frac 1r} \left(\int_0^\infty |F(x,t)|^{r'} d\omega_a(x)\right)^{\frac 1{r'}} dt 
\\
& \le ||T_a^\star T_a(F)||_{L^q_t L^r_a} ||F||_{L^{q'}_t L^{r'}_a} \le \overline C(a,r) ||F||^2_{L^{q'}_t L^{r'}_a},
\end{align*}
where in the second to the last inequality we have used one more H\"older inequality in $t$, and in the last we have used Theorem \ref{T:comp}. This proves the theorem with $C'(a,r) = \overline C(a,r)^{1/2}$.

\end{proof}

We can finally provide the

\begin{proof}[Proof of Theorem \ref{T:GV}]
The proof follows from Theorem \ref{T:striT} by a duality argument. Let $\vf\in \sn$, then for every $F\in \snn$, we have
\begin{align*}
|\sa\sa T_a^\star(\vf),F\da\da_{L^2_t L^2_a}| & = |\sa \vf,T_a(F)\da_{L^2_a}| \le ||\vf||_{L^2_a} ||T_a(F)||_{L^2_a}
\\
& \le C'(a,r) ||\vf||_{L^2_a} ||F||_{L^{q'}_t L^{r'}_a}.
\end{align*}    
Taking the supremum on all $F\in \snn$, we reach the conclusion
\begin{equation}\label{tri}
||T_a^\star(\vf)||_{L^{q}_t L^{r}_a} \le C'(a,r) ||\vf||_{L^2_a}.
\end{equation}
Keeping in mind that, in view of \eqref{solBB}, the solution to \eqref{cp0} with $F=0$  is given by
\[
u(x,t) = S_a(t)\vf(x) = T_a^\star(\vf)(x,t),
\]
 it is clear that \eqref{tri} is the same as \eqref{striuno}. This proves the theorem when $F\equiv 0$. The non-homogeneous case, now follows by arguing as in \cite{Stri}. Define
\[
v(x,t) = \int_0^t \int_0^\infty S_a(x,y,t-\tau) F(y,\tau) d\omega_a(y) d\tau = \int_0^t S_a(t-\tau)(F(\cdot,\tau))(x) d\tau.
\]
Then, Proposition \ref{P:disp} gives 
\begin{align*}
& ||v(\cdot,t)||_{L^r_a} \le \int_0^t ||S_a(t-\tau)(F(\cdot,\tau))||_{L^r_a} d\tau \le C(a,r)\int_0^t \frac{||F(\cdot,\tau))||_{L^{r'}_a}}{|t-\tau|^{(a+1)(\frac 12 - \frac 1r)}} d\tau
\\
& \le  C(a,r)\int_0^\infty \frac{||F(\cdot,\tau))||_{L^{r'}_a}}{|t-\tau|^{(a+1)(\frac 12 - \frac 1r)}} d\tau.
\end{align*}
This estimate is similar to \eqref{TstarTT}, and arguing as in the end of the proof of Theorem \ref{T:comp}, we reach the desired conclusion that
\[
||v||_{L^q_t L_a^{r}} \le \overline C(a,r) ||F||_{L^{q'}_t L^{r'}_a}.
\]
 
\end{proof}

Next, we give the 

\begin{proof}[Proof of Theorem \ref{T:restriction}]
By Theorem \ref{T:striT} and the Plancherel identity \eqref{planrad} applied with $a = 2\nu+1$
we have for $F\in \snn$ 
\begin{equation}\label{res1}
\int_0^\infty |\mathcal H_{\frac{a-1}2}(T_a(F))(\xi)|^2 d\omega_a(\xi) = ||T_a(F)||^2_{L^2_a} \le C'(a,r) ||F||^2_{L^{q'}_t L^{r'}_a}.
\end{equation}
Next, \eqref{T} in Lemma \ref{L:T} gives
\begin{equation}\label{T2}
\mathcal H_{\frac{a-1}2}(T_a(F))(\xi) = \int_\R \mathcal H_{\frac{a-1}2}(S_a(-t)(F(\cdot,t)))(\xi) dt.
\end{equation}
We now apply \eqref{hank}, which gives
\[
\mathcal H_{\frac{a-1}2}(S_a(-t)(F(\cdot,t)))(\xi) = e^{i t x^2} \mathcal H_{\frac{a-1}2}(F(\cdot,t))(\xi). 
\]
Substituting this expression in the right-hand side of \eqref{T2}, we find
\begin{equation}\label{T3}
\mathcal H_{\frac{a-1}2}(T_a(F))(\xi) = \int_\R e^{i t \xi^2} \mathcal H_{\frac{a-1}2}(F(\cdot,t))(\xi) dt = \widetilde{\mathcal H}_{\frac{a-1}2}(F)(\xi,-\frac{\xi^2}{2\pi}),
\end{equation}
where $\widetilde{\mathcal H}_{\frac{a-1}2}(F)$ denotes the Fourier-Hankel transform of $F$, see \eqref{FH}. Substituting \eqref{T3} in \eqref{res1} we finally obtain \eqref{re1}.

\end{proof}


\section{The case $-1<a<0$: weighted estimates}\label{S:weighted}

In this section we analyze the case $-1<a<0$. To handle the two different regimes in the estimate \eqref{eli2} in Proposition \ref{P:disp}, we will need the following generalization of the Hardy-Littlewood-Sobolev theorem, see e.g. \cite[Lemma 5.1]{BG}. 

\begin{lemma}\label{HLS}
Let $0<\gamma_{1}, \gamma_2 <1$, $C_1,C_2>0$. Let $K:\mathbb{R}\to \mathbb{R}$ be such that
\begin{align*}
|K(t)|\le \begin{cases}
\frac{C_1}{|t|^{\gamma_{1}}} \ \ \ \ \text{if }\ |t|\le 1,
\\
\frac{C_2}{|t|^{\gamma_{2}}} \ \ \ \ \text{if}\ |t|\ge1.
\end{cases}
\end{align*}
If $1<p_{1}<q_{1}<\infty$ , $1< p_{2}, q_{2} <\infty$
\begin{align*}
\gamma_{1}=1+\frac{1}{q_{1}}-\frac{1}{p_{1}}, \ \ \ \text{and}\ \ \ \ \gamma_{2}\ge1+\frac{1}{q_{2}}-\frac{1}{p_{2}},
\end{align*}  
then one has
\begin{align}\label{meglio}
\|f\star K\|_{L^{q_{1}+q_2}}\le C \|f\|_{L^{p_{1}\cap p_{2}}}.
\end{align}
\end{lemma}
 
We can now establish the following substitute of Theorem \ref{T:comp}.

\begin{theorem}\label{T:comp2}
Let $-1<a< 0$ and for $r>2$ assume that $q$ and $q_\infty$ satisfy the conditions
\begin{equation}\label{wa2}
\frac 2q = \frac 12 - \frac 1r,\ \ \ \ \ \ \ \ \frac 2{q_\infty} \le (a+1)\left(\frac 12 - \frac 1r\right).
\end{equation}
There exists $C(a,r,q_\infty)>0$ such that for every $F$, such that $F k_1^{1-\frac{2}{r'}}\in L^{q'\cap q'_\infty}_t L^{r'}_a$,
one has the following estimate
\begin{equation}\label{strimas}
||\ ||T_a^\star T_a(F)(\cdot,t) k_1^{1-\frac 2r}||_{L^r_a}\ ||_{L^{q+q_\infty}_t } \le C(a,r,q_\infty) ||\ ||F(\cdot,t) k_1^{1-\frac{2}{r'}}||_{L^{r'}_a}\ ||_{L^{q'\cap q'_\infty}_t}.
\end{equation}
\end{theorem}

\begin{proof}
Keeping Definition \ref{D:mn} in mind, we obtain from \eqref{stelle}
\[
T_a^\star T_a(F)(x,t) k_1(x)^{1-\frac 2r} =  \int_\R S_a(t-\tau)(F(\cdot,\tau))(x) k_1(x)^{1-\frac 2r}  d\tau.
\]
This gives
\begin{align}\label{wo}
& ||T_a^\star T_a(F)(\cdot,t) k_1^{1-\frac 2r}||_{L^r_a} \le \int_\R ||S_a(t-\tau)(F(\cdot,\tau)) k_1^{1-\frac 2r}||_{L^r_a}  d\tau
\\
& \le C(a,r) \int_\R \left\{|t-\tau|^{-(a+1)(\frac 12 - \frac 1r)} + |t-\tau|^{-(\frac{1}2-\frac 1r)}\right\} ||F(\cdot,\tau) k_1^{1-\frac{2}{r'}}||_{L^{r'}_a} d\tau
\notag
\\
& = \int_\R K(t-\tau) h(\tau)d\tau,
\notag
\end{align}
where in the second inequality we have used \eqref{eli2} in Proposition \ref{P:disp}, and in the third we have set
\[
h(\tau) = ||F(\cdot,\tau) k_1^{1-\frac{2}{r'}}||_{L^{r'}_a},\ \ \ \ \ K(t) = \left\{|t|^{-(a+1)(\frac 12 - \frac 1r)} + |t|^{-(\frac{1}2-\frac 1r)}\right\}.
\]
Since $0<a+1<1$, it is clear that we have 
\[
K(t) \le \begin{cases}
C |t|^{-(\frac 12 - \frac 1r)}, \ \ \ \text{if}\ \ 0<|t|<1,
\\
C |t|^{-(a+1)(\frac 12 - \frac 1r)}, \ \ \ \text{if}\ \ |t|\ge 1.
\end{cases}
\]
For $r>2$ we now apply Lemma \ref{HLS} with 
\[
\gamma_1 = \frac 12 - \frac 1r,\ \ \ \ \ \ \ \ \gamma_2 = (a+1)(\frac 12 - \frac 1r)<\gamma_1.
\]
Then for $q$ and $q_\infty$ as in \eqref{wa2}, from \eqref{wo}  
we obtain the desired conclusion \eqref{strimas}
for some $C(a,r,q_\infty)>0$.

\end{proof}

We next use Theorem \ref{T:comp2} to establish the following substitute of Theorem \ref{T:striT}.

\begin{theorem}[Weighted restriction estimates]\label{T:striT2}
Under the assumptions of Theorem \ref{T:comp2}, for every $F\in \snn$ we have
\begin{equation}\label{strimas2}
||T_a(F)||_{L^2_a} \le C(a,r,q_\infty)  ||\ ||F(\cdot,t) k_1^{1-\frac{2}{r'}}||_{L^{r'}_a}\ ||_{L^{q'\cap q'_\infty}_t}.
\end{equation}
\end{theorem}

\begin{proof}
We argue as in the proof of Theorem \ref{T:striT}, except that now we bring to the centerstage the weight $k_1(x)$, as follows 
\begin{align*}
||T_a(F)||^2_{L^2_a} & = \sa T_a(F),T_a(F)\da_{L^2_a} = \sa\sa T_a^\star(T_a(F))k_1^{1 -\frac 2r},F k_1^{\frac 2r-1}\da\da_{L^2_t L^2_a} 
\\
& \le \int_\R ||T_a^\star T_a(F)(\cdot,t) k_1^{1-\frac 2r}||_{L^r_a}\ ||F(\cdot,t) k_1^{1-\frac{2}{r'}}||_{L^{r'}_a} dt,
\end{align*}
where in the last inequality we have used H\"older with respect to the variable $x$ and the measure $d\omega_a$.
If we now use the duality between the spaces $L^{q_t +q_\infty}_t$ and $L^{q' \cap q'_\infty}_t$, we obtain from the latter estimate 
\[
||T_a(F)||^2_{L^2_a} \le C(a,r,q_\infty)\ ||\ ||T_a^\star T_a(F)(\cdot,t) k_1^{1-\frac 2r}||_{L^r_a}\ ||_{L^{q+q_\infty}_t }\ \  ||\ ||F(\cdot,t) k_1^{1-\frac{2}{r'}}||_{L^{r'}_a}\ ||_{L^{q'\cap q'_\infty}_t}.
\]
At this point, we invoke \eqref{strimas} in Theorem \ref{T:comp2} to reach the desired conclusion \eqref{strimas2}.

\end{proof}

We are now able to provide the 

\begin{proof}[Proof of Theorem \ref{T:GV2}]
Let $\vf\in \sn$, then for every $F\in \snn$, we have
\begin{align*}
|\sa\sa T_a^\star(\vf),F\da\da_{L^2_t L^2_a}| & = |\sa \vf,T_a(F)\da_{L^2_a}| \le ||\vf||_{L^2_a} ||T_a(F)||_{L^2_a}
\\
& \le   C(a,r,q_\infty)  ||\ ||F(\cdot,t) k_1^{1-\frac{2}{r'}}||_{L^{r'}_a}\ ||_{L^{q'\cap q'_\infty}_t}\ \ ||\vf||_{L^2_a},
\end{align*}    
where in the last line we have used \eqref{strimas2} from Theorem \ref{T:striT2}. Applying this inequality to the function $G(x,t) = F(x,t) k_1(x)^{1-\frac{2}{r'}}$, we obtain
\[
|\sa\sa T_a^\star(\vf) k_1^{\frac{2}{r'}-1},G\da\da_{L^2_t L^2_a}| \le C(a,r,q_\infty)  ||\ ||G(\cdot,t)||_{L^{r'}_a}\ ||_{L^{q'\cap q'_\infty}_t}\ \ ||\vf||_{L^2_a}.
\]
Taking the supremum on all $F\in \snn$, and using the duality between the spaces $L^{q_t +q_\infty}_tL^r_a$ and $L^{q' \cap q'_\infty}_tL^{r'}_a$, we find (using the observation $\frac 2{r'}-1 = 1-\frac 2r$)
\[
||\ ||T_a^\star(\vf)(\cdot,t) k_1^{1-\frac 2r}||_{L^r_a}\ ||_{L^{q+q_\infty}_t }\le C(a,r,q_\infty)  ||\vf||_{L^2_a}.
\]
Keeping in mind that 
\[
u(x,t) = S_a(t)\vf(x) = T_a^\star(\vf)(x,t),
\]
we finally reach the conclusion \eqref{striuno}.
This proves the theorem when the forcing term $F\equiv 0$. The non-homogeneous case $F\not= 0$ follows by arguing as in the end of the proof of Theorem \ref{T:GV}. 
 
\end{proof}


\section{Well-posedness}\label{S:wp}

In this section we consider the nonlinear problem \eqref{nonlin}. Our main objective is to use the results in Section \ref{S:stri} to prove global or local in time existence and uniqueness of the solution, i.e. well-posedness,  in the regime $a\ge 0$, see the precise statement in Theorem \ref{T:main}.

Before moving to the precise definition of solution for the Cauchy problem \eqref{nonlin}, we show that it    conserves  the mass.  Indeed we have the following lemma.

\begin{lemma}\label{L:mass}
Suppose $\mu = i b$, with $b\in \R$. Then the  $L^2_a$-norm of the solution of \eqref{nonlin} is time-invariant, namely
\begin{equation}\label{mass}
\frac{d}{dt}\int_0^\infty  |u(x,t)|^2 d\omega_a(x)=0.
\end{equation}
\end{lemma}
\begin{proof}
We have 
\begin{align*}
& \frac{d}{dt}\int_0^\infty  |u|^2 d\omega_a= 2 \Re \int_0^\infty \bar u \p_t u  d\omega_a = 2 \Re\ i \int_0^\infty \left[\bar u \Ba u + b |u|^{p+1}\right] d\omega_a
\\
& = 2 \Re\ i \int_0^\infty \left[b |u|^{p+1} - |\p_x u|^2\right] d\omega_a = 0,
\end{align*}
where in the second equality we have used the PDE in \eqref{nonlin}, and in the third one we have used \eqref{pos}.

\end{proof}
We will use Lemma \ref{L:mass} later to move from a local in time well-posedness to a global one.
Next we give  the definition  for solutions to the Cauchy problem \eqref{nonlin}. This definition is inspired by Proposition \ref{P:rep}. 
\begin{definition}\label{solution}
We say that a function $u$ is solution to the Cauchy problem \eqref{nonlin} if it satisfies the identity 
\begin{equation}\label{sol1}
u(x,t)=\int_0^\infty S_a(x,y,t) \vf(y) d\omega_a(y) + \mu \int_0^t \int_0^\infty S_a(x,y,t-\tau) |u(y,\tau)|^{p-1}u(y,\tau) d\omega_a(y) d\tau.
\end{equation}
\end{definition} 
 
We now  use appropriate scalings to classify our problem \eqref{nonlin} as either a $L^2_a$-critical or subcritical problem.
If $u$ is a solution to the nonlinear PDE in \eqref{nonlin}, then it is easy to check that $v(x,t) = \la^{-\gamma} u(\frac{x}{\la},\frac{t}{\la^2})$ is also a solution if and only if $\gamma = \frac{2}{p-1}$. Therefore, if we define   
\begin{equation}\label{rescal}
u_\lambda(x,t)=\lambda^{-\frac{2}{p-1}}u(\lambda^{-1}x,\lambda^{-2}t),\ \ \ \ \ \ \ \la>0,
\end{equation}
then $u_\lambda$ is solution to \eqref{nonlin} with initial data $\phi_\lambda(x)=\lambda^{-\frac{2}{p-1}}\phi(\lambda^{-1}x)$. Since
$$\|\phi_\lambda\|_{L^2_a}=\lambda^{-\frac{2}{p-1}+\frac{a+1}{2}}\|\phi\|_{L^2_a},$$
it is clear that $\vf_\la$ has the same $L^2_a$-norm of $\vf$ if and only if 
\[
p=1+\frac{4}{a+1}.
\]
These considerations motivate the following.

\begin{definition}\label{D:crit}
We say that the Cauchy problem \eqref{nonlin} is $L^2_a$-\emph{critical} if $p=1+\frac{4}{a+1}$. If $p<1+\frac{4}{a+1}$, then we say that it is $L^2_a$-\emph{subcritical}.
\end{definition} 

\vskip 0.2in

We are now ready to provide the
 
\vskip 0.2in

\begin{proof}[Proof of Theorem \ref{T:main}]  The method of proof is by now standard and it is based on a contraction method, see Kato's original work \cite{Kato} and \cite{Caze, MuS1} for further references. 

Assume first that the problem is $L^2_a$-critical, and therefore $p = 1 + \frac{4}{a+1}$. Since we will use the Strichartz estimates \eqref{striuno} in Theorem \ref{T:GV}, we fix a $a$-admissible pair $(q,r)$, and denote by $C = C(a,r)>0$ the corresponding universal constant in the estimate. Moreover, we indicate with $\tilde C = \tilde C(\mu,a,r)>0$ the constant appearing in \eqref{00}. We emphasize that the pair $(q,r)$ will not be changed throughout the discussion. Since we have defined our solution as a  function  that satisfies \eqref{sol1}, we are led to define the operator on  $X$  
\begin{equation}\label{Phi}
\Phi(v):= \int_0^\infty S_a(x,y,t) \vf(y) d\omega_a(y) + \mu \int_0^t \int_0^\infty S_a(x,y,t-\tau) |v(y,\tau)|^{p-1}v(y,\tau) d\omega_a(y) d\tau.
\end{equation}
Our objective is to show that, if we assume that $\|\phi\|_{L^2_a}$ is small enough, precisely, 
\begin{equation}\label{smallphi}
\|\phi\|_{L^2_a} \le \left(\frac{1}{|\mu|(2C)^p}\right)^{\frac{1}{p-1}},
\end{equation}
and take
\begin{equation}\label{rho}
\rho=2C\|\phi\|_{L^2_a},\ \ \ \text{and} \ \ \ \ \rho\le (4\tilde C)^{-\frac{1}{p-1}},
\end{equation}
if we consider the closed ball $B_\rho = \{u\in X\mid ||u||_X \leq  \rho\}\subset X$, then $\Phi:B_\rho\to B_\rho$ and 
for any $v, w \in B_\rho$ we have 
\begin{equation}\label{rest}
\|\Phi(v)-\Phi(w)\|_{X}\leq \frac12 \|v-w\|_{X}.
\end{equation}
Therefore, $\Phi$ is a contraction, hence it admits a unique fixed point $u = \Phi(u)\in B_\rho$. 

With this goal in mind, using \eqref{striuno} we have 
\begin{equation}\label{step1}\|\Phi(v)\|_{L^q_t L^r_a}\leq C\|\phi\|_{L^2_a}+C|\mu| \||v|^p\|_{L^{q'}_t L^{r'}_a}
=C\|\phi\|_{L^2_a}+C|\mu| \|v\|_{L^{pq'}_t L^{pr'}_a}^p.
\end{equation}
We now note that, if we set
$\tilde q = pq',\quad \tilde r = pr'$, 
then the pair $(\tilde q, \tilde r)$ is $a$-admissible, i.e.
\begin{equation}\label{tildes}
\frac{2}{\tilde q} = (a+1)\left(\frac 12 - \frac{1}{\tilde r}\right).
\end{equation}
In fact, since $p = 1 + \frac{4}{a+1}$, and \eqref{tildes} is equivalent to having 
$$\frac{2}{ q'}=(a+1)\left(\frac{p}{2}-\frac{1}{r'}\right),$$
we see that \eqref{tildes} is true if and only if 
\[
2 - \frac 2q = \frac{a+1}2\left(1+\frac{4}{a+1}\right) - (a+1)\left(1-\frac 1r\right),
\]
and this easily reduces to the $a$-admissibility of the pair $(q,r)$.
Hence, first taking in \eqref{step1} the supremum on all $a$-admissible pairs $(\tilde q,\tilde r)$, and then on the pairs $(q,r)$, we obtain  
\begin{equation}\label{step2}
\|\Phi(v)\|_{X}\leq C\|\phi\|_{L^2_a}+C|\mu|\|v\|^p_{X}.\end{equation}
For $\rho$ as in \eqref{rho} and $v\in B_\rho$, we now have
\[
C|\mu|\|v\|^p_{X}\le C |\mu|\rho^p \le \frac{\rho}2\ \Longleftrightarrow\ 2 C |\mu|\rho^{p-1}\le 1\ \Longleftrightarrow\ 2 C |\mu|(2C\|\phi\|_{L^2_a})^{p-1}\le 1.
\] 
Therefore, if we assume that $\|\phi\|_{L^2_a}$ satisfy \eqref{smallphi}, 
we infer that indeed $2 C |\mu|(2C\|\phi\|_{L^2_a})^{p-1}\le 1$, and since from the first equation in \eqref{rho} we have $C\|\phi\|_{L^2_a} = \rho/2$, we conclude  
from \eqref{step2} that
$$\|\Phi(v)\|_{X}\leq \rho/2+\rho/2=\rho,$$
hence $\Phi:B_\rho\rightarrow B_\rho$. We next show that \eqref{rest} holds. For this, we observe that, with $F(z) = \mu |z|^{p-1}z$, $z\in \C$, we have
\begin{equation}\label{diff0}
\Phi(v)-\Phi(w)= \int_0^t \int_0^\infty S_a(x,y,t-\tau) (F(v) - F(w))(y,\tau) d\omega_a(y) d\tau.
\end{equation}
From \eqref{diff0} and the  Strichartz estimates  \eqref{striuno} again, we have
\begin{equation}\label{diffvw}
\|\Phi(v)-\Phi(w)\|_{L^q_t L^r_a}\leq C ||F(v) - F(w)||_{L^{q'}_t L^{r'}_a},
\end{equation}
where $C>0$ is the same constant as in \eqref{smallphi}.
Observing that there exists $C(p)>0$ such that
\[
|F(v) - F(w)| \le |\mu| C(p)\left[|v|^{p-1} + |w|^{p-1}\right] |v-w|,
\]
by an application of H\"older inequality, we obtain 
\begin{align*}
& ||F(v) - F(w)||_{L^{q'}_t L^{r'}_a} \le |\mu| C(p)\left(\int_\R ||\left[|v|^{p-1} + |w|^{p-1}\right] |v-w|\ ||^{q'}_{L^{r'}_a} dt\right)^{1/{q'}}
\\
& = |\mu| C(p)\left(\int_\R \left(\int_{\R^+} \left[|v|^{p-1} + |w|^{p-1}\right]^{r'} |v-w|^{r'} d\omega_a(y)\right)^{\frac{q'}{r'}} dt\right)^{1/{q'}}
\\
& \le |\mu| C(p)\left(\int_\R \left[\left(\int_{\R^+} \left[|v|^{p-1} + |w|^{p-1}\right]^{p' r'} d\omega_a(y)\right)^{1/{p'r'}} \left(\int_{\R^+} |v-w|^{pr'} d\omega_a(y)\right)^{1/{pr'}}\right]^{q'} dt\right)^{1/{q'}}
\\
& \le C(\mu,r,p) \left(\int_\R \left[||v||^{p-1}_{L^{pr'}_a} + ||w||^{p-1}_{L^{pr'}_a}\right]^{q'} ||v-w||^{q'}_{L^{pr'}_a}dt\right)^{1/{q'}}
\\
& \le  C(\mu,r,a) \left(\int_\R \left[||v||^{p-1}_{L^{pr'}_a} + ||w||^{p-1}_{L^{pr'}_a}\right]^{p'q'}dt\right)^{1/{p'q'}} \left(\int_\R  ||v-w||^{pq'}_{L^{pr'}_a}dt\right)^{1/{pq'}},
\end{align*}
where in the last inequality we have underscored the fact that, being $p = p(a)$, we can write $C(\mu,r,p) = C(\mu,r,a)$.
Since as we have shown the pair $\tilde q = pq',\quad \tilde r = pr'$ is $a$-admissible, taking suprema in the latter inequality, and substituting in \eqref{diffvw}, we finally obtain with $\tilde C = C(\mu,r,a) C$,
\begin{equation}\label{00}
\|\Phi(v)-\Phi(w)\|_{L^q_t L^r_a} \leq  \tilde C(|\|v\|^{p-1}_{X}+|\|w\|^{p-1}_{X})|\|v-w\|_{X}.
\end{equation}
If we now keep in mind that $v$ and $w$ belong to the ball $B_\rho$, we obtain from this estimate
\[
\|\Phi(v)-\Phi(w)\|_{L^q_t L^r_a}\leq 2 \tilde  C\rho^{p-1}|\|v-w\|_{X}.
\]
Since by \eqref{rho} we have $2 \tilde  C\rho^{p-1}\le \frac 12$, the desired conclusion \eqref{rest} follows. This shows that $\Phi$ is a contraction on $B_\rho$, hence we have a unique fixed point $u \in  B_\rho$. As a consequence, such $u$ is a unique global in time solution.    

We now want to show that $u\in L^\infty_tL^2_a$. Using the fact that $u = \Phi(u)$, where $\Phi(u)$ is as in \eqref{Phi}, and the trivial observation $|F(u)| \le |\mu| |u|^p$, we obtain from Proposition \ref{P:uni} for any fixed $t$
\begin{equation}\label{infty}
\|u(\cdot,t)\|_{L^2_a}\leq \|\phi\|_{L^2_a}+C\int_0^\infty\left(\int_0^\infty |u(x,\tau)|^{2p} d\omega_a(x)\right)^{1/2}d\tau=\|\phi\|_{L^2_a}+C\|u\|_{L^p_tL^{2p}_a}^p.
\end{equation}
Now we make the crucial observation that, since
\[
\frac{2}{p}=(a+1)\left(\frac{1}{2}-\frac{1}{2p}\right)\ \Longleftrightarrow\  p=1+\frac{4}{a+1},
\]
the fact that the problem \eqref{nonlin} is $L^2_a$-\emph{critical} implies that the pair $(p,2p)$ be $a$-admissible. Since $u\in B_\rho\subset X$, taking first the supremum on all $a$-admissible pairs in \eqref{infty}, and then the supremum in $t\in \R$, we thus find
\begin{equation}\label{useful}
\sup_{t\in \R_+}\|u(\cdot,t)\|_{L^2_a}\leq \|\phi\|_{L^2_a} + C\|u\|^p_{X}.
\end{equation}  
Finally, we prove that the solution $u$ is continuous in time into the space $L^2_a$. We first show that the linear part is continuous, namely 
\begin{equation}\label{contL}
\lim_{|t-\tau|\rightarrow 0}\|S_a(t)\phi-S_a(\tau)\phi\|_{L^2_a}=0.
\end{equation}
By \eqref{hank} we have 
$$\mathcal H_{\frac{a-1}2}( S_a(t) \vf -S_a(\tau) \vf)(x)=(e^{-itx^2}-e^{-i\tau x^2})\mathcal H_{\frac{a-1}2}( \vf)(x).$$
Combined with the Plancherel type formula \eqref{planrad}, this gives
\begin{eqnarray*}
\lim_{|t-\tau|\rightarrow 0}\|S_a(t)\phi-S_a(\tau)\phi\|_{L^2_a}^2&=&\lim_{|t-\tau|\rightarrow 0}\int_0^\infty |\mathcal H_{\frac{a-1}2}( S_a(t) \vf(x) -S_a(\tau) \vf (x) )|^2d\omega_a\\
&=&
\lim_{|t-\tau|\rightarrow 0}\int_0^\infty |e^{-itx^2}-e^{-i\tau x^2}|^2|\mathcal H_{\frac{a-1}2}( \vf)(x)|^2d\omega_a=0,
\end{eqnarray*}
where in the last equality we have used dominated convergence.
For the non-homogeneous part it is enough to prove that for $ t\leq \tau$
$$\lim_{|t-\tau|\rightarrow 0}\int_t^\tau \||u(\cdot,s)|^{p-1}u(\cdot,s)\|_{L^2_a}ds=0,$$
but this comes from the fact that, as proved in  \eqref{useful},
$$\int_0^\infty \||u(\cdot,s)|^{p-1}u(\cdot,s)\|_{L^2_a}ds\leq \|\phi\|_{L^2_a} + \|u\|_{X}^p.$$

\medskip
We now consider the subcritical case, namely $1<p<1 + \frac{4}{a+1} = \frac{a+5}{a+1}$.  The argument to show contraction is similar, but with some slight difference. Let us go back to \eqref{step1}. In the subcritical case it is no longer true that $(pq',pr')$ is $a$-admissible, if such is the pair $(q,r)$. If in fact we let as before $\tilde r = p r'$, and we consider $\tilde q$ for which $(\tilde q,\tilde r)$ is $a$-admissible, then since $\frac p2 < \frac 12 + \frac{2}{a+1}$, we must have 
\begin{align*}
\frac{2}{\tilde q} & = (a+1)\left(\frac 12 - \frac{1}{\tilde r}\right) =  \frac{a+1}p\left(\frac p2 - \frac{1}{r'}\right)< \frac{a+1}p\left(\frac 12 + \frac{2}{a+1} - 1 + \frac1r\right)
\\
&  = - \frac{a+1}p\left(\frac 12 - \frac 1r\right) + \frac 2p = - \frac{2}{pq} +  \frac 2p = \frac{2}{pq'},
\end{align*}
which shows that $\tilde q> pq'$. If we are on a finite time-interval $[0,T]$, then H\"older inequality gives
\[
\left(\frac 1T \int_0^T ||v||^{pq'}_{L^{pr'}_a} dt\right)^{1/{pq'}} \le \left(\frac 1T \int_0^T ||v||^{\tilde q}_{L^{pr'}_a} dt\right)^{1/{\tilde q}},
\]
which implies  
$$\|v\|_{L^{pq'}_{[0,T]} L^{pr'}_a}^p\leq T^\delta \|v\|_{L^{\tilde q}_{[0,T]} L^{\tilde r}_a}^p\leq  T^\delta \|v\|_{X_T}^p,$$
for $\delta = \frac{1}{pq'}-\frac{1}{\tilde q} >0$.
With $\Phi(v)$ defined as in \eqref{Phi}, and $C>0$ as in the Strichartz estimate \eqref{striuno} in Theorem \ref{T:GV}, we can thus replace \eqref{step2} by 
\begin{equation}\label{step3}
\|\Phi(v)\|_{X_{T}}\leq C\|\phi\|_{L^2_a}+C|\mu|T^{\delta}\|v\|^p_{X_T},
\end{equation}
for some $\delta>0$. We stress that now $C = C(a,\tilde r) = C(a,r,p)$.
As in the $L^2_a$-critical case, we set the radius of the ball $B_\rho = \{u\in X_T\mid ||u||_{X_T} \leq  \rho\}$ to be $\rho=2C\|\phi\|_{L^2_a}$ (see \eqref{rho}), but we now use the time $T$ to contract. Namely, as long as 
\begin{equation}\label{time}
C|\mu| T^\delta \rho^p = C|\mu|T^\delta(2C\|\phi\|_{L^2_a})^{p-1}\leq 1/2,
\end{equation}
we infer from \eqref{step3} that $ \|\Phi(v)\|_{X_{T}}\leq \rho$, and similarly, by possibly restricting $T$ even further, we can replace \eqref{rest} with 
$$\|\Phi(v)-\Phi(w)\|_{X_{T}}\leq \frac 12 \|v-w\|_{X_T}.
$$
These estimates imply, as before, the existence of a unique fixed point 
\[
u = \Phi(u)\in B_\rho\subset X_T.
\]
We claim that $u$ is actually the unique solution in the whole space $X_T$. Suppose, in fact, there is another solution 
$v\in X_T$. Given $\gamma>0$, for $t\in [0,\gamma]$. We have from \eqref{sol1},
\begin{equation}\label{diff0}
u-v=\Phi(u)-\Phi(v)= \mu\int_0^t \int_0^\infty S_a(x,y,t-\tau) (|u|^{p-1}u- |v|^{p-1}v)(y,\tau) d\omega_a(y) d\tau.
\end{equation}
Proceeding as above we can estimate  
$$\|u-v\|_{X_\gamma}\leq  C\gamma^{\delta}(\|v\|^{p-1}_{X_T} +   \|u\|^{p-1}_{X_\gamma}) \|u-v\|_{X_\gamma}\leq C\gamma^{\delta}(\|v\|^{p-1}_{X_T} +  
 \rho^{p-1}) |\|u-v\|_{X_\gamma}.$$
Let now $\gamma$ be such that $C\gamma^{\delta}(\|v\|^{p-1}_{X_T} +  
 \rho^{p-1})<1/2$, then we have that $u=v$ almost everywhere in the strip $\R\times[0,\gamma]$. By iterating $\cong T/\gamma$ times, we conclude $u=v$ a.e. in $\R\times[0,T]$.

We now pass to global well-posedness under the assumption that in the Cauchy problem  \eqref{nonlin} the constant $\mu$ is purely imaginary. Note that \eqref{time} shows that  $T=T(\|\phi\|_{L^2_a})$. At this point, we solve a similar Cauchy problem in $\R\times[T,2T]$, with initial datum $\vf_1 = u(\cdot,T)$. Since by the conservation of mass in Lemma \ref{L:mass}, we have $\|u(\cdot,T)\|_{{L^2_a}}= \|\phi\|_{L^2_a}$, we can advance up to time  $2T$. As a consequence, for any fixed arbitrary time interval $[0,M]$, with about $M/T$ steps we obtain a unique solution on $[0,M]$.  

\end{proof}


\section{Comparison with other Schr\"odinger equations on the half-line}\label{S:comp}

In the literature there exist several differential operators on the half-line, each one of them being treated by ad hoc methods. The aim of this section is to discuss the connection between the Bessel operator \eqref{Ba} and two other models.  
\subsection{The Kimura model from population biology}\label{EM} In \cite{EM} Epstein and Mazzeo consider the operator on $(0,\infty)$
\begin{equation}\label{Lb}
\mathscr L_b = y \frac{\p^2}{\p y^2} + b \frac{\p}{\p y},\ \ \ \ \ \ b>0,
\end{equation}
and they construct the heat kernel for the semigroup $e^{-t \mathscr L_b}$ under Feller's zero flux condition
\begin{equation}\label{zeroflux}
\underset{y\to 0^+}{\lim}\ y^b \p_y v(y,t) = 0.
\end{equation} 
The model \eqref{Lb} is the simplest prototype of a Kimura operator and it is related, via a change of variable, to the Wright-Fisher diffusion process associated with a class of degenerate parabolic equations from population biology. In \cite[Prop. 5.1]{Gams} it was observed that the change of variable
\begin{equation}\label{tran}
y = \frac{x^2}4,\ \ \ \ \ \ \ \ \ \ \ b = \frac{a+1}2,
\end{equation}
sends in a one-to-one, onto fashion solutions of \eqref{cp0} with $a>-1$ into solutions of \eqref{Lb}. One has in fact
\begin{equation}\label{BaLb}
u_t(x,t) - \Ba u(x,t) = v_t(y,t) - \mathscr L_b v(y,t).
\end{equation}
Furthermore, one easily verifies that
\begin{equation}\label{nanb}
x^a \p_x u(x,t) = 2^{2b-1} y^b \p_y v(y,t).
\end{equation}
Therefore, the Neumann condition in \eqref{cp0} is transformed into \eqref{zeroflux}. As a consequence, the results in this paper provide corresponding ones for solutions of the Cauchy problem in $\R^+$
\[
\p_t v - i \mathscr L_b v = G(y,t),\ \ \ \ \ \ v(y,0) = \psi(y),
\]
under the Neumann condition \eqref{zeroflux}.

\subsection{The inverse square potential}\label{Hardy} The most studied model in quantum mechanics is the scaling critical Schr\"odinger operator $H = \Delta - \frac{c}{|x|^2}$. Higher-dimensional Strichartz estimates for this operator were proved in \cite{BPST1, BPST2, FFFP}. We also mention the works \cite{GS, RS} for optimal borderline dispersive estimates for a generic potential $V(x)$ on the whole line and in $\R^3$. On the half-line $\R^+$ the operator $H$ is often written  
\begin{equation}\label{Hm}
\mathcal H_\ell v = \p_{xx} v - \frac{\ell(\ell+1)}{x^2} v,
\end{equation}
where $\ell$ represents angular momentum. It is well-known that $\mathcal H_\ell$ is limit-circle at $x=0$ if and only if $-\frac 32 <\ell < \frac 12$, see \eqref{vs} and \cite{Se}. Therefore 
\begin{equation}\label{ssa}
\mathcal H_\ell\ \  \text{is self-adjoint in}\ \  L^2(\R+)\ \Longleftrightarrow\  (\ell+\frac 12)^2 \ge 1\ \Longleftrightarrow\ \ell\in (-\infty,-\frac 32] \cup [\frac 12,\infty).
\end{equation}
 The transformation 
\begin{equation}\label{uv}
u(x) = x^{-\frac{a}2} v(x),
\end{equation}
converts $\Ba$ into the Hardy operator \eqref{Hm} (see \cite[p.562]{Ze} and also \cite[1.15.7, p.65]{BGL} for a nice account of how to remove a drift), with $a$ and $\ell$ connected by the equation
\[
\frac a2(\frac a2 -1) = \ell(\ell+1).
\]
One has in fact the notable identity
\begin{equation}\label{remove}
x^{\frac{a}2} \Ba u =  \p_{xx} v - \frac a2\left(\frac a2-1\right)\frac{v}{x^2},
\end{equation}  
which shows that solutions of $\Ba u = 0$ go into solutions of $\mathcal H_\ell v = 0$ if and only if
\begin{equation}\label{ella}
(\ell+\frac 12)^2 = (\frac{a-1}2)^2\ \ \ \Longleftrightarrow\ \ \ \ell = \frac a2-1\ \ \  \text{or}\ \ \ \ell = - \frac a2.
\end{equation}
In particular, if $\ell = \frac a2-1$, the  solutions $u_1(x) = 1$ and $u_2(x) = x^{1-a}$ of \eqref{Ba} respectively go into the linearly independent ones of $\mathcal H_\ell$
\begin{equation}\label{vs}
v_1(x) = x^{\ell+1},\ \ \ \ \ \ \ v_2(x) = x^{-\ell}.
\end{equation}
When $\ell = - \frac a2$ this correspondence is reversed.
Note that \eqref{uv} establishes a unitary map between $L^2_a$ and $L^2(\R^+)$, and therefore in view of \eqref{ella} we have: \eqref{sa}\ $\Longleftrightarrow$\ \eqref{ssa}.
When 
\begin{equation}\label{range}
- \frac 32 < \ell < \frac 12,
\end{equation}
both solutions $v_1$ and $v_2$ belong to $L^2$ near $x=0$, therefore boundary conditions must be imposed in order to achieve self-adjointness. The literature on the spectral properties of $\mathcal H_\ell$ is quite vaste. A short list of references is \cite{EK, BDG, AB, DG, GPS, GNZ}. Observe that the transformation \eqref{uv} converts the problem \eqref{cp0} into 
\begin{equation}\label{cp00}
\begin{cases}
\p_t v - i \mathcal H_\ell v = G(x,t),\ \ \ \ \ x\in \R^+,\ t\in \R,
\\
\underset{x\to 0^+}{\lim} x^{\frac{a}2-1}\left[x \p_x v(x,t) - \frac{a}{2} v(x,t)\right] = 0,\ \ \ \ t\in \R,
\\
v(x,0) = \psi(x),
\end{cases}
\end{equation}
where $G(x,t) = x^{\frac{a}2} F(x,t)$, $\psi(x) = x^{\frac{a}2} \vf(x)$. We also note that the boundary condition in \eqref{cp00} expresses the vanishing of the Wronskian $W[x^{\frac a2},v]$ as $x\to 0^+$. Depending on the value of $\ell$ in \eqref{ella}, one has
\begin{equation}\label{uffa} 
\begin{cases}
\underset{x\to 0^+}{\lim} W[x^{\ell+1},v] = \underset{x\to 0^+}{\lim} x^{\ell}\left[x \p_x v(x,t) - (\ell+1) v(x,t)\right] = 0,\ \ \ \ \ \ \ \ \ \text{when}\ \ \ \ell = \frac a2 -1,
\\
\underset{x\to 0^+}{\lim} W[x^{-\ell},v] = \underset{x\to 0^+}{\lim}  x^{-\ell-1}\left[x \p_x v(x,t) + \ell v(x,t)\right] = 0,\ \ \ \ \ \ \ \ \ \text{when}\ \ \ \ell = - \frac a2.
\end{cases}
\end{equation}
We have the following.
\begin{proposition}\label{P:janus}
The kernel $S_a(x,y,t)$ in \eqref{SaS0}, associated with the problem \eqref{cp0}, and the kernel $H_\ell(x,y,t)$, associated with the problem \eqref{cp00} when $\ell = - \frac a2$, are linked by the equation:
\begin{equation}\label{conne}
S_a(x,y,t) = (xy)^{-\frac a2} H_\ell(x,y,t).
\end{equation}  
\end{proposition}

\begin{proof}
We note explicitly that the boundary condition that we are imposing is the second one in \eqref{uffa}. To prove \eqref{conne}, suppose that $\psi\in L^2(\R^+)$ be given, and consider $\vf = x^{-\frac a2} \psi$.
Since one obviously has
\[
||\psi||_{L^2(\R^+)} = ||\phi||_{L^2_a},
\]
and since, as we have already observed, the function $u(x,t) = x^{-\frac a2} v(x,t)$ solves \eqref{cp0}, we have
\[
u(x,t) = \int_0^\infty S_a(x,y,t) \vf(y) y^a dy.
\]
This formula can be rewritten as
\[
v(x,t) = x^{\frac a2} \int_0^\infty S_a(x,y,t) \psi(y) y^{\frac a2} dy = \int_0^\infty H_{\ell}(x,y,t) \psi(y) dy,
\]
if $\ell = - \frac a2$ and $H_{\ell}(x,y,t) = (xy)^{\frac a2} S_a(x,y,t)$.

\end{proof}

We mention that, using the spectral analysis of $\mathcal H_\ell$, 
in \cite[Lemma 3.1]{HKT} Holzleitner, Kostenko \& Teschl computed the kernel of the operator $e^{i t \mathcal H_\ell}$, under the second boundary condition in \eqref{uffa}, for $\ell$ in the range \eqref{mezzi}
\begin{equation}\label{mezzi}
- \frac 12 < \ell < \frac 12.
\end{equation}
Since in this framework one must use the case $\ell = - \frac a2$ in \eqref{uffa}, note that the range \eqref{mezzi} corresponds to our $-1<a<1$. For such range, their formula
\begin{equation}\label{uffa2}
H_\ell(x,y,t) = \frac{e^{i\frac{\pi}2(\ell-\frac 12)}}{2t} (xy)^{1/2} J_{-\ell-\frac 12}(\frac{xy}{2t})\ e^{i\frac{x^2+y^2}{4t}},
\end{equation}
is in perfect accordance with that obtained by substituting our kernel \eqref{SaS0} in \eqref{conne} in Proposition \ref{P:janus}, i.e.,
\begin{equation}\label{Hl}
H_\ell(x,y,t) = (xy)^{\frac a2} S_a(x,y,t),\ \ \ \ \ \ \ell = - \frac a2.
\end{equation}
Combining \eqref{Hl} with \eqref{notsogoodone} in Lemma \ref{L:dis} we obtain the following dispersive estimate for the group $e^{itH_\ell}$ in the range $0<\ell<1/2$:
\begin{equation}\label{disHl}
|| e^{itH_\ell} \psi||_{L^\infty(\R^+, dn_\ell)} \le C(\ell) \left\{|t|^{\ell - \frac 12} + |t|^{-\frac{1}2}\right\} ||\psi||_{L^1(\R^+,dm_\ell)}, \ \ \ t\not= 0,
\end{equation}
where we have defined 
\[
dn_\ell(x) = \min\{1,x^\ell\} dx,\ \ \ \ \text{and}\ \ \ \ \  dm_\ell(x) = \max\{1,x^{-\ell}\} dx.
\]
The estimate \eqref{disHl} should be compared to \cite[(1.8), p.770]{HKT}.
If instead we combine \eqref{Hl} with \eqref{good} in Lemma \ref{L:dis}, we obtain for the range $-\frac 12<\ell \le 0$
\begin{equation}\label{goody}
||e^{it H_\ell}||_{L^1(\R^+)\to L^\infty(\R^+)}  \le C |t|^{-1/2},\ \ \ \ |t|\not= 0.
\end{equation}
In connection with these results, we also cite the more recent works \cite{KTT, GHT}. Finally, we mention that, again using spectral analysis, in \cite[(3.23), p.175]{KoT} Kova\v{r}\'ik \& Truc computed the kernel of $e^{it \mathcal H_\ell}$  under the former of the boundary conditions in \eqref{uffa} (this corresponds to taking $\ell = \frac a2 -1$). They have no restrictions on the value of $\ell$, but since they consider the Friedrichs extension, they are imposing a Dirichlet boundary condition. In the unweighted case $s=0$ of their \cite[Theor. 2.4]{KoT} they established the following dispersive estimate
\begin{equation}\label{KoT}
||e^{it H_\ell}||_{L^1(\R^+)\to L^\infty(\R^+)}  \le C |t|^{-1/2},\ \ \ \ |t|\not= 0.
\end{equation}


\vskip 0.2in


\section{Appendix}\label{S:app}

Consider the Bessel equation of order $\nu\in \mathbb C$
\begin{equation}\label{besseleq}
z^2 \frac{d^2 J}{dz^2} + z \frac{dJ}{dz} + (z^2 - \nu^2)J = 0.
\end{equation}
Assume $\nu\not=\mathbb N\cup\{0\}$. Then, two linearly independent solutions of \eqref{besseleq} are the Bessel function of the first kind and order $\nu$ 
\begin{equation}\label{besseries}
J_\nu(z) = \sum_{k=0}^\infty \frac{(-1)^k}{\G(k+1) \G(\nu+k+1)} \left(\frac z2\right)^{2k+\nu}\ \ \ \ \ \ \ \ |\arg z|<\pi,
\end{equation}
and the function 
\begin{equation}\label{negbes}
J_{-\nu}(z)=\sum_{k=0}^\infty \frac{(-1)^k}{\G(k+1) \G(k+1-\nu)} \left(\frac z2\right)^{2k-\nu}\ \ \ \ \ \ \ \ |\arg z|<\pi.
\end{equation}
The Wronskian of the functions $J_\nu$ and $J_{-\nu}$ satisfies the equation (see \cite[(5.9.1), p. 113]{Le})
\begin{equation}\label{wronJ}
W\left[J_\nu,J_{-\nu}\right] = - \frac{2\sin(\nu \pi)}{\pi z},
\end{equation}
which shows that these functions are linearly independent if $\nu\not\in \mathbb Z$.
It is well-known that
\begin{equation}\label{half}
J_{-1/2}(x) = \sqrt{\frac{2}{\pi x}} \cos x,\ \ \ \ \ J_{1/2}(x) = \sqrt{\frac{2}{\pi x}} \sin x.
\end{equation}
From \eqref{besseries}, \eqref{negbes} we have as $z\to 0$, $|\arg z|<\pi$,
\begin{equation}\label{Js}
J_{\nu}(z) \cong \frac 1{\G(\nu+1)} \left(\frac{z}{2}\right)^{\nu},\ \ \ \ J_{-\nu}(z) \cong \frac 1{\G(1-\nu)} \left(\frac{z}{2}\right)^{-\nu}.
\end{equation}
The asymptotic behavior for large $z$ is much more delicate. One has for $\Re\nu>-\dfrac12$, and $0<\delta<\pi$
\begin{align}\label{jnuinfty} 
J_\nu(z)&=\sqrt{\frac2{\pi
z}}\cos\left(z-\frac{\pi\nu}2-\frac\pi4\right)+
O(z^{-\frac32})\\
&\quad\text{as }|z|\to\infty,\quad-\pi+\delta<\arg z<\pi-\delta.
\notag
\end{align}
Consider now the generalized Bessel equation 
\begin{equation}\label{genbessel}
z^2 \Psi''(z) + (1 - 2\alpha) z \Psi'(z) + \left[\beta^2 \gamma^2
z^{2\gamma} + (\alpha^2 - \nu^2 \gamma^2)\right] \Psi(z) = 0,
\end{equation}
see \cite[(5.4.11), p.106]{Le}.
If $J(z)$ is a solution to 
\eqref{besseleq}, then the function defined by the
transformation
\begin{equation}\label{besselcv}
\Psi(z) = z^\alpha J(\beta z^\gamma)
\end{equation}
solves the equation \eqref{genbessel}. Another solution to the equation \eqref{besseleq} is given by the Bessel function of the second kind
\begin{equation}\label{Y}
Y_\nu(z) =  \frac{J_\nu(z) \cos(\nu\pi) - J_{-\nu}(z)}{\sin(\nu\pi)}.
\end{equation}
From \eqref{Js}, \eqref{Y}, we see that
\begin{equation}\label{Ys}
z^\nu Y_{\nu}(z)\ \underset{z\to 0}{\longrightarrow}\ - \frac {2^\nu}{\sin(\nu\pi)\G(1-\nu)} = - \frac {2^\nu\G(\nu)}{\pi},
\end{equation}
where in the last equality we have used \cite[(1.2.2), p.3]{Le}.
We need the following result.

\begin{proposition}\label{P:ef}
Assume that $a>-1$. For $\kappa>0$ the equation $\Ba f = - \kappa f$ on $\R^+$ admits the linearly independent solutions
\begin{equation}\label{ass}
f_1(x) = x^{\frac{1-a}2} J_{\frac{a-1}2}(\sqrt \kappa x),\ \ \ \ \ \ \emph{and}\ \ \ \ f_2(x) = x^{\frac{1-a}2} Y_{\frac{a-1}2}(\sqrt \kappa x).
\end{equation} 
Only $f_1$ satisfies the Neumann condition $x^a f_1'(x)\to 0$ as $x\to 0^+$.
\end{proposition}

\begin{proof}
We write the equation $\Ba f = - \kappa f$ in the form 
\[
x^2 f''(x) + a x f'(x) + \kappa x^2 f(x) = 0.
\]
A comparison with \eqref{genbessel} shows that we must choose
\[
\alpha = \frac{1-a}2,\ \ \ \ \gamma = 1,\ \ \ \ \beta = \sqrt \kappa,\ \ \ \ \ \nu =  \frac{a-1}2 > - 1.
\]
With such value of $\nu$, the Bessel functions $J_{\nu}(x)$, and $Y_\nu(x)$
 solve the equation \eqref{besseleq}, and since their Wronskian is (see \cite[(5.9.2), p.113]{Le})
\[
W[J_\nu(x),Y_\nu(x)] = \frac{2}{\pi x},
\]
they are linearly independent on $\R^+$. By \eqref{besselcv} we infer that two linearly independent solutions of $\Ba f = - \kappa f$ are given by \eqref{ass}. 
Since $f_1(x) = x^{-\nu} J_\nu(\sqrt \kappa x)$, $f_2(x) = x^{-\nu} Y_\nu(\sqrt \kappa x)$, using \cite[(5.3.5), p.103 \& (5.4.9), p.105]{Le}, which give
\[
\frac{d}{dz}\left[z^{-\nu} J_\nu(z)\right] = - z^{-\nu} J_{\nu+1}(z),\ \ \ \ \frac{d}{dz}\left[z^{-\nu} Y_\nu(z)\right] = - z^{-\nu} Y_{\nu+1}(z),
\]
and keeping in mind that $a = 2\nu+1$, we find
\[
x^a f_1'(x) = x^{2\nu+1} O(x^{-\nu} J_{\nu+1}(\sqrt \kappa x)) = O(x^{2(\nu+1)})\ \longrightarrow\ 0,
\]
since $\nu>-1$. On the other hand, we have from \eqref{Ys}
\[
|x^a f_2'(x)|\ \longrightarrow\ \gamma > 0.
\]

\end{proof}

There is another important pair of linearly independent solutions to \eqref{besseleq}, the Hankel functions, or Bessel functions of the third kind. They are defined by the formulas
\begin{equation}\label{Hank}
H^{(1)}_\nu(z) = J_\nu(z) + i Y_\nu(z),\ \ \ \ \ \ \ H^{(2)}_\nu(z) = J_\nu(z) - i Y_\nu(z),\ \ \ \ \ \ |\arg z|<\pi.
\end{equation} 
The Wronskian of the Bessel and Hankel functions is
\begin{equation}\label{wronhank}
W[J_\nu(z),H^{(2)}_\nu(z)] = - \frac{2i}{\pi z},\ \ \ \  \ \ \ \ W[H^{(1)}_\nu(z),H^{(2)}_\nu(z)] = - \frac{4i}{\pi z}.
\end{equation}
We thus obtain corresponding linearly independent solutions to the equation $\Ba f = - \kappa f$ on $\R^+$
\begin{equation}\label{ass2}
g_1(x) = x^{\frac{1-a}2} H^{(1)}_{\frac{a-1}2}(\sqrt \kappa x),\ \ \ \ \ \ \emph{and}\ \ \ \ g_2(x) = x^{\frac{1-a}2} H^{(2)}_{\frac{a-1}2}(\sqrt \kappa x).
\end{equation}






\bibliographystyle{amsplain}

\end{document}